\documentclass[11pt,oneside]{amsart}

\usepackage{ amsmath, amssymb, amsthm, hyperref, graphicx,   verbatim, enumerate, color}
 \usepackage{mathabx}
\usepackage[paper=a4paper, marginpar=2.4cm]{geometry}
\usepackage[all]{xy}
\usepackage[utf8]{inputenc}

\newcommand{\C}{\mathbb{C}}
\newcommand{\bb}{\mathbb{B}}

\newcommand{\dd}{\mathbb{D}}
\newcommand{\Z}{\mathbb{Z}}
\newcommand{\N}{\mathbb{N}}
\newcommand{\pp}{\mathbb{P}}
\newcommand{\e}{\varepsilon}

 \newcommand{\cW}{\mathcal{W}}

\newcommand{\fr}{\partial}
\newcommand{\om}{\Omega}
\newcommand{\set}[1]{\left\{#1\right\}}
\newcommand{\norm}[1]{\left\Vert#1\right\Vert}
\newcommand{\abs}[1]{\left\vert#1\right\vert}
\newcommand{\cd}{{\C^2}}

\newcommand{\pu}{{\mathbb{P}^1}}

\newcommand{\rest}[1]{ \arrowvert_{#1}}

\newcommand{\unsur}[1]{\frac{1}{#1}}

\newcommand{\lrpar}[1]{\left(#1\right)}

\newcommand{\loc}{\mathrm{loc}}

\newcommand{\inv}{^{-1}}

\DeclareMathOperator{\Int}{Int}

\DeclareMathOperator{\jac}{Jac}
\DeclareMathOperator{\comp}{Comp}
\DeclareMathOperator{\diam}{Diam}
\DeclareMathOperator{\core}{Core}

\newtheorem{prop}{Proposition} [section]
\newtheorem{thm}[prop] {Theorem}
\newtheorem{lem}[prop] {Lemma}
\newtheorem{cor}[prop]{Corollary}

\newtheorem*{mainthm}{Main Theorem}

\newtheorem{question}[prop]{Question}
\newtheorem{conjecture}[prop]{Conjecture}

\theoremstyle{definition}
\newtheorem{defi}[prop] {Definition}

\theoremstyle{remark}

\newtheorem{rmk}[prop]{Remark}

\begin{document}


\title[Structure of hyperbolic maps]{Structure of   hyperbolic polynomial automorphisms of $\cd$ with disconnected Julia sets}
\author{Romain Dujardin}
\address{Sorbonne Universit\'e, Laboratoire de probabilit\'es, statistique et mod\'elisation, UMR 8001,  4 place Jussieu, 75005 Paris, France}
\email{romain.dujardin@sorbonne-universite.fr}

\author{Mikhail Lyubich}
\address{Institute for Mathematical Sciences Stony Brook University Stony Brook, NY 11794}
\email{mlyubich@math.stonybrook.edu}

\begin{abstract}
For a hyperbolic polynomial automorphism of $\C^2$ with a disconnected Julia set,
and under a mild dissipativity condition, 
we give a topological description of the components of the Julia set.
Namely, there are finitely many ``quasi-solenoids" that govern the asymptotic behavior of the 
orbits of all non-trivial
components.   This can be viewed as a refined Spectral Decomposition for a hyperbolic 
map, as well as a two-dimensional version of the (generalized) Branner-Hubbard theory 
 in     one-dimensional polynomial dynamics. An important geometric ingredient of the theory
is  a John-like property of the Julia set in the unstable leaves. 
\end{abstract}

 \maketitle

\setcounter{tocdepth}{1}
\tableofcontents

\section{Introduction}

\subsection{Preamble on hyperbolic dynamics} 
The classical {\em Spectral Decomposition}  of a  
hyperbolic (Axiom A) real diffeomorphism $f$ of a compact manifold  
(developed by Smale, Anosov, Sinai, Bowen, and others)
provides us with a rather complete topological picture of its dynamics.
Namely, the non-wandering set  $\Omega(f)$
is decomposed into finitely many {\em basic sets},
each of which modeled on an  irreducible Markov chain.
Among these basic sets there are several {\em attractors} that
govern the asymptotic behavior of generic points of the manifold. 
This picture has become a prototype for numerous other settings,
including one-dimensional,  non-invertible, holomorphic, partially or  non-uniformly 
hyperbolic dynamical systems.

In the context of complex polynomial automorphisms of $\cd$,
hyperbolic maps arise naturally as perturbations of one-dimensional
hyperbolic polynomials. They were first studied in the late
1980s by Hubbard and Oberste-Vorth \cite{HOV1, HOV2} who showed that their  
topological structure  can be fully described in terms of the
original one-dimensional  maps, whose Julia set and  attracting cycles
get perturbed to the basic sets of $f$ 
(see also Forn{\ae}ss-Sibony~\cite{fornaess-sibony}).

Computer experiments indicate that, though hyperbolicity is not a
prevalent phenomenon in dimension two,  there should 
exist still plenty of non-perturbative examples.
The first such  candidate (a quadratic H\'enon map with two co-existing
attracting cycles) was  proposed by Hubbard; it was further investigated  
  by Oliva in his thesis~\cite{oliva}. 
However, it is a challenging problem, which  requires
computer assistance, to prove the  hyperbolicity of a particular example,
and this one  still remains unconfirmed.  
Some time later, Ishii   justified the hyperbolicity of several other non-perturbative H\'enon maps:
see \cite{ishii-nonplanar, ishii_survey, ishii_new_one} (of course,  along
with each such example comes an open set of hyperbolic parameters).

A systematic theory of hyperbolic polynomial automorphisms of $\cd$
was launched by Bedford and Smillie in the early 1990's , relying notably on 
 methods from  Pluripotential Theory. In particular, they showed in~\cite{bs1} that
any such a map   only has one non-trivial  basic set, its Julia set
$J(f)$, while all others are just attracting cycles. Further combinatorial study of hyperbolic Hénon maps was carried out by Ishii and Smillie~\cite{ishii-smillie}. 

In this paper we will reveal a  finer structure of the Julia set, related to its
connected components, that leads to a finer ``spectral
decomposition''.
Namely, under mild dissipativity assumptions, we will show that there are finitely many {\em
  quasi-solenoids} that govern the  asymptotic behavior of all non-trivial 
components. Some of these quasi-solenoids are {\em tame} (i.e. lie on
the boundary of the  basins of  some attracting cycles),
while others might be {\em queer} (we do not know whether they 
actually exist). 

Let us  conclude this preamble by suggesting a potentially  important role that
hyperbolic maps may play in the H\'enon story.   They are not only
interesting simple models for the general non-uniformly hyperbolic situation,
but they may  also be seen as  ``germs'' for a  Renormalization Theory  which would
lead to self-similarity features of the parameter spaces. In this respect, 
  renormalizing hyperbolic H\'enon maps around quasi-solenoids would be the
beginning of this story. 

\subsection{One-dimensional prototype} 

Understanding the topological structure of the Julia set is one of the most basic problems in holomorphic 
dynamics. For polynomials in one variable, Fatou and Julia proved that 
 the connectivity properties of the Julia set are dictated by the 
dynamical behavior of critical points. When the critical points do not escape, 
the Julia set $J$ is connected; on the contrary, if all critical points do escape, $J$ is a Cantor set.  
If  $J$ is connected and locally connected,  the theory of external rays 
of Douady and Hubbard~\cite{douady-hubbard} and the theory of geodesic laminations of Thurston~\cite{thurston} 
give a topological 
model for the Julia  set as   the quotient of the circle by an equivalence relation which records   
 the landing pattern of external rays.
 When the Julia set of a polynomial is disconnected, it admits uncountably many components,
 and one challenge is to characterize when a component is non-trivial (i.e. not a point) in terms of 
the induced dynamics on the set of components. It turns out that this happens when and only when
this component is preperiodic to a component containing a critical point: this is due to 
Branner and Hubbard~\cite{branner-hubbard} 
for cubic polynomials, and Qiu and Yin~\cite{qiu-yin} in the general
case (based upon the Kahn-Lyubich machinery \cite{KL1, KL2}). 
Then one may describe non-trivial periodic components by realizing them as 
Julia sets of connected polynomial-like maps and using the
Douady and Hubbard Straightening Theorem~\cite{douady-hubbard2}.

In the hyperbolic case, the above theory is much easier and
had belonged to folklore of the field: 

\begin{thm}\label{thm:folklore}
Let $p$ be a hyperbolic polynomial in $\C$, with a disconnected Julia set. 
Then the filled Julia set $K$ 
has uncountably many components, and only countably of them are non-trivial. Any non-trivial component 
is preperiodic, and there are finitely many periodic components, each of which 
containing an attracting periodic point. 
\end{thm}
   
Note that this is really a statement about polynomials: there are examples of hyperbolic rational maps on $\pu$
whose Julia sets are  Cantor sets of circles~\cite{mcmullen_buried}.

\subsection{Main result}

In this article we address similar issues in the setting of polynomial automorphisms of $\cd$. 
Let $f$ be a polynomial automorphism of $\cd$ with non-trivial dynamics: by this we mean for instance that the algebraic 
degree of the  iterates $f^n$ tend to infinity  
(see below \S\ref{subs:vocabulary} for more details on this).
Its Julia set $J=J_f$ is the  set of points at which  both $(f^n)_{n\geq 0}$ and $(f^{-n})_{n\geq 0}$ are not 
locally normal.  
We also classically denote by $K^+$ (resp. $K^-$), the set of points with bounded forward 
(resp. backward) orbits, $K = K^+\cap K^-$ and $J^\pm = \fr K^\pm$, so that $J = J^+\cap J^-$. 
The complex Jacobian $\jac f$ is a non-zero constant. Thus, replacing $f$ by $f\inv$ if necessary, without  loss of generality 
we assume from now on that $\abs{\jac f}\leq 1$. 

In this context, the connected vs. disconnected dichotomy for the Julia set 
was studied by Bedford and Smillie~\cite{bs6}, who proved that the 
connectedness of $J$,  or equivalently of $K$,  is equivalent to the  non-existence  of 
``unstable critical points'', which are   defined as tangencies
between certain dynamically defined foliations.
(Recall that $f$ has no critical point in the usual sense,
but these unstable critical points
play the same role as escaping critical points in dimension one.)
Bedford and Smillie  also showed 
that when $J$ is connected, there is a well-defined family of external rays along unstable manifolds, parameterized by a 
``solenoid at infinity'', which is the inverse limit of the 
dynamical system defined by $z\mapsto z^d$ on the unit circle. 

To proceed further and try to extend the Douady-Hubbard description of the Julia set in terms of the 
combinatorics of external rays,   given our  current state of knowledge,
we need    
to assume that $f$ is uniformly hyperbolic. Recall 
from~\cite{bs1} that $f$ is said to be \emph{hyperbolic} if $J$ is a hyperbolic set, which must then 
be  of saddle type. In this  case, 
$f$ satisfies Smale's Axiom A in $\cd$, and the Fatou set is the union
of finitely many basins of attraction.
(See \cite{ishii_survey} 
for an introductory account to  this topic, which also   discusses some combinatorial/topological 
  models for Julia sets.)
  
By using the convergence of unstable external rays, it was shown 
in~\cite{bs7} that if $f$ is hyperbolic and $J$ is connected, then $J$ can be described as a finite 
quotient of the solenoid at infinity. A non-trivial consequence of the results of~\cite{bs5},\cite{bs6} 
and~\cite{bs7} is that in this case $f$ cannot be conservative, that is,  $\abs{\jac f}<1$ (see~\cite[Cor. A.3]{bs7}; recall that we assume $\abs{\jac f}\leq 1$ here). 
An alternate  argument for this fact
was given by the first-named author in~\cite{connex}, where it is shown 
that a hyperbolic automorphism $f$ with connected    Julia set
  must possess an attracting periodic point, so in particular $\abs{\jac f}<1$. Surprisingly 
  enough, the existence of an 
  attracting point does not seem to follow easily from the description of $J$ as a quotient of the solenoid. 

In this article we focus on the disconnected case. A motivating question is the following conjecture from~\cite{connex}. 

\begin{conjecture}\label{conj:cantor}
Let $f$ be a dissipative and hyperbolic automorphism of $\cd$, without attracting points. 
Then $J$ is a Cantor set. 
\end{conjecture}

Our main result is an essentially complete generalization of Theorem~\ref{thm:folklore}
 in two dimensions, under a mild dissipativity assumption. 

 \begin{mainthm}
 Let $f$ be a hyperbolic polynomial automorphism of $\cd$, with a disconnected Julia set, and such that 
 $\abs{\jac f}< 1/\deg f$. Then there are uncountably many  components of 
 $J$, which  can be of three (mutually exclusive) types: 
 \begin{enumerate} 
 \item point;
 \item leafwise bounded;
 \item or quasi-solenoid.
 \end{enumerate}
 Quasi-solenoidal components are periodic and there are only finitely many of them. 
Any component of type (2) is wandering and converges
  to a quasi-solenoidal one under forward iteration. The 
components of $K$ are classified accordingly. 
  
 Under an additional assumption (NDH) on the behavior of stable holonomy  
 between components, 
 any quasi-solenoidal component of $K$ contains an attracting periodic point. 
 \end{mainthm}

Here $\deg f$ refers to the dynamical degree of $f$, which is the growth rate of algebraic degree under iteration
 (see \S\ref{subs:vocabulary}). By definition, a component of $J$ 
 is  \emph{leafwise bounded} if it is 
  a relatively bounded subset of some unstable manifold; this implies that its topology is that of a full plane continuum, properly embedded in $\C^2$.   
 A \emph{quasi-solenoid} is a    connected component  
 with local product structure, which is 
  totally disconnected   in the stable direction 
   and   locally connected  and 
  leafwise unbounded in the 
  unstable direction (see Definition~\ref{defi:quasi_solenoid}). Components of type
  (2) are analogous to strictly preperiodic components in dimension 1;
  note however that by the local product structure of $J$ there are
  uncountably many of them. Countability is restored by saturating with semi-local stable manifolds (see Theorem~\ref{thm:component}). 
 The meaning of the  (NDH) assumption will be explained below.

 \subsection{Outline}
Let us discuss some of the main ideas of the proof, which occupies the most part  of the paper. 
First, the assumption on the Jacobian is used to guarantee 
that \emph{the slices of $J$ (resp. $K$) by stable manifolds are totally disconnected.}
It is reminiscent of the stronger \emph{substantial dissipativity} assumption $\abs{\jac f}< 1/(\deg f)^2$ 
used in \cite{tangencies, lyubich-peters_classification, lyubich-peters}. 
We could indeed use substantial dissipativity and Wiman's Theorem in the style of these papers
 to achieve  stable total disconnectivity. However, 
  hyperbolicity allows for  a Hausdorff dimension calculation  
which gives a better bound on the  Jacobian (see Section~\ref{sec:total disconnected}). 

The key step of the finiteness property in the main theorem is an analysis of 
 geometry of the unstable slices of  $J$ and $K$. Using external rays, we first show in Section~\ref{sec:john} that the complement of $K$ 
 along unstable manifolds satisfies a weak version of the \emph{John property}. 
 This property  implies that the components of 
 $K\cap W^u$  are locally connected, and 
that  locally there are only 
   finitely many components of diameter bounded from below.
 
 This finiteness is used to get  a classification of \emph{semi-local components} 
 of $J^+$ and $K^+$. By this we mean that we fix a large bidisk $\bb$ (in adapted coordinates)
  in which $J^+$ and $K^+$ are vertical-like
  objects, and we look at components of $J^+\cap \bb$ (resp. $K^+\cap \bb$). 
 We prove that \textit{these semi-local components behave like
   components of $J$ (resp. $K$) for
   one-dimensional   
 polynomials}: only countably many of them are non-trivial, that is, not reduced to vertical 
 submanifolds, and  any 
 non-trivial such component is preperiodic. Besides the finiteness induced by the John-like property, this relies 
 on a key {\em homogeneity property}  of such a semi-local component: 
 either all its unstable slices are ``thin'', or all of them are ``thick''. 
 To prove this  {\em thin-thick  dichotomy}   
 we show 
  that if a semi-local component admits a thin unstable slice, then by a careful choice of $\bb$ we can arrange that the stable foliation of this semi-local component is transverse to $\fr\bb$. It follows that this component 
  has    a global product structure in $\bb$   (see Section~\ref{sec:semi_local} for details).
 
If $C$ is a non-trivial component of $J$, it is easy to see that 
the $\omega$-limit set of $C$  must be contained in one of the finitely many 
 thick semi-local components of $J^+$.  We show that it must have local product structure, 
 hence be a quasi-solenoidal component of $J$. The main step is the following: 
 for large $m\neq n$, by the expansion in the unstable direction, the unstable slices of 
  $f^m(C)$ and $f^m(C)$ have a diameter bounded from below, so 
 if $x_n\in f^n(C)$ is close to $x_m\in f^m(C)$, 
  by the finiteness given by the John-like property,
 $f^n(C)$ and $f^m(C)$ must correspond one to the other under local stable 
 holonomy. Furthermore, such a quasi-solenoidal component must coincide with the limit set of its semi-local component in $J^+$, and the finiteness of the number of attractors follows (see Section~\ref{sec:components}). 
 
To get a complete generalization of the one-dimensional situation, it remains to show that such a 
quasi-solenoidal component must ``enclose'' some attracting periodic point. 
Unfortunately, all our attempts 
towards this result stumbled over the following issue: if $x,y\in J$ are such that $y\in W^s(x)$, 
the stable holonomy 
induces a local homeomorphism $J\cap W^u_\loc(x)\to J\cap W^u_\loc(y)$. The point is that it 
might not be the case   in general that 
this local homeomorphism can be continued along paths in  $J\cap W^u(x)$, even when  $J\cap W^u(x)$ is a relatively bounded subset of $W^u(x)$. (Compare with the Reeb phenomenon for foliations, illustrated in  Figure~\ref{fig:reeb}.)
This is a well-known difficulty in hyperbolic dynamics, 
which was encountered for instance  in the classification of Anosov diffeomorphisms 
(see \S\ref{subs:NDH} for a short discussion). If this continuation property holds 
-- this is the  Non-Divergence of Holonomy (NDH) property referred to in the main theorem-- then we 
can indeed conclude  that non-trivial periodic components of $K$ contain attracting orbits (see Section~\ref{sec:NDH}, in particular Theorem~\ref{thm:no_queer_components}). 
This yields in particular a conditional proof of Conjecture~\ref{conj:cantor}. Let us also note 
 that a simple instance where the NDH property holds is when the stable 
 lamination of $J^+$ is transverse to $\fr\bb$ 
(for some choice of $\bb$), a property which can be checked in
practice on  specific  
examples. 
 
In the course of the paper, we also establish a 
 number of complementary facts, which do not enter into the proof of the main theorem: the existence of an 
 external ray landing at every point of $J$ (see Theorem~\ref{thm:access}); the 
 structure of attracting basins (see \S~\ref{subs:basins}); a simple topological model for the dynamics on Julia components (see \S~\ref{subs:branched}); the topological transitivity of quasi-solenoids (see Theorem~ \ref{thm:complement_mixing}). In Appendix~\ref{app:core} 
 we sketch the construction of the 
\emph{core} of a quasi-solenoidal component, which aims at describing   
its  topological structure. 
 
\subsection*{Notes and acknowledgments.} 
Some of these results were already announced at the conference ``Analytic Low-Dimensional Dynamics'' 
in Toronto in June 2019. We are grateful to Pierre Berger for
pointing out  
Proposition~\ref{prop:berger} to us. 
The second-named author was partially supported by an NSF grant,
Hagler and Clay Fellowships. Part of this work was carried out
during his visits  
of  the Hagler Institute for Advanced Study at Texas A\&M,  
the Center of Theoretical Studies at ETH Zürich, and MSRI at Berkeley.
We thank these institutions for their generous support.

\section{Preliminaries and notation}\label{sec:preliminaries}
  
  \subsection{Vocabulary of complex Hénon maps}\label{subs:vocabulary}
  
 If $\bb = D\times D$ is a bidisk, we denote by $\fr^vB = \fr D\times D$ (resp. $\fr^hB =   D\times \fr D$) the vertical (resp. horizontal) boundary.  An object in $\bb$ is horizontal if it intersects $\fr\bb$ only in $\fr^v\bb$, 
 and likewise for vertical objects.
 A closed horizontal submanifold is a branched cover of finite degree over the first projection. 

Let us collect some  standard facts and notation (see \cite{fm, bs1, bls, fornaess-sibony}). If
  $f$ is a polynomial diffeomorphism of $\cd$ with non-trivial dynamics, then by 
  making a polynomial change of coordinates we may assume that 
$f$ is a composition of complex Hénon mappings 
$(z,w)\mapsto (p_i(z)+a_iw,a_iz)$.  
 In particular $\deg(f^n)   = (\deg f)^n$ for every $n\geq 0$. 
We fix such coordinates from now on. As it is customary in this area of research, 
we will often abuse terminology and simply  refer to $f$ as  a \emph{complex Hénon map}.
The degree of $f$  is $d=\prod \deg(p_i)\geq 2$ and the relation $\deg(f^n) = d^n$ 
holds so that $d$ coincides with the so-called \emph{dynamical degree} of $f$.  

In these adapted coordinates, there exists $R>0$ such that for  the bidisk 
$\bb:=D(0, R)^2$, we have that $f(\bb)\cap\bb$ (resp. $f\inv(\bb)\cap \bb$)
 is horizontally (resp. vertically) contained in $\bb$ and the points of 
 $\fr^v(\bb)$ (resp. $\fr^h(\bb)$ escape under forward (resp. backward) iteration. 
 
\begin{itemize}
\item $K^\pm$ is the set of points with bounded forward orbits under $f^{\pm1}$ 
 and $K= K^+\cap K^-$. Note that $K^+$ is vertical in $\bb$ and 
 $f(\bb\cap K^+)\subset K^+$. Similarly, $K^-$ is horizontal and $f\inv(\bb\cap K^-)\subset K^-$. 
\item $J^\pm = \fr K^\pm$ are the forward and backward Julia sets. 
If $f$ is dissipative then  $K^- = J^-$. 
\item $J = J^+\cap J^-$ is the Julia set. 
\end{itemize}

Following \cite{bs6}, we say that $f$ is \emph{unstably disconnected} if for some (and hence any) saddle 
periodic point $p$, $W^u(p)\cap K^+$ admits a compact component (relative to the topology induced 
by the biholomorphism $W^u(p)\simeq \C$), and unstably connected otherwise. If 
$f$ is unstably disconnected,   then it admits   an \emph{unstable transversal} $\Delta^u$, 
that is a relatively compact domain in $W^u(p)$  which is a 
horizontal submanifold in $\bb$: indeed   pick a bounded 
Jordan domain $U\subset W^u(p)$ containing a compact component of $W^u(p)\cap K^+$ such that 
$\fr U \cap K^+ = \emptyset$ and iterate  it forward.

  \subsection{Hyperbolicity and local product structure} \label{subs:hyperbolicity}
  Throughout the paper we assume that $f$ is hyperbolic on $J$ (hence Axiom A on 
  $\cd$ by \cite{bs1}), with hyperbolic splitting $T\cd\rest{J}   = E^u\oplus E^s$. 
Then there exists  a continuous  Riemannian metric $\abs{\cdot}$ on $J$ and constants $s<1<u$ such that 
for any $x\in J$, and any $v\in E^u(x)\setminus \set{0}$, $\abs{Df_x \cdot v}\geq u\abs{v}$ (resp. for 
any $v\in E^s(x)$, $\abs{Df_x \cdot v}\leq  s\abs{v}$). By~\cite{saddle}, it is enough 
to assume that $f$ is hyperbolic on $J^\star$, where $J^\star$ is the closure of saddle periodic points 
(and a posteriori one deduces that  $J = J^\star$). 

In this situation the local stable and unstable manifolds of points of $J$ have local uniform geometry: there exists a uniform $r>0$ such that for every $x\in J$, $W^u(x)$ (resp. $W^s(x)$  \emph{is of size $r$ at $x$}, in the sense that it 
 contains a graph of slope at most 1 over a disk of radius $r$ in $E^u(x)$ (resp. $E^s(x)$). 
  The reader is referred to \cite{hyperbolic, topological} for a detailed study of this notion. We denote by 
  $W^{s/u}_\delta(x)$ the local stable/unstable manifold of radius $\delta$ at $x$, which is by definition the 
    component of $W^{s/u}(x)$ in 
  $B(x, \delta)$. When the precise size does not matter, we simply denote them  
  by $W^{s/u}_\loc$. Slightly reducing the expansion constant $u$ if necessary, given two points $z, z'$ in 
  some local unstable manifold
  $W^u_\delta (x)$, there is a uniform constant $C$ such that  $d(f^{-n}(z), f^{-n}(z'))\leq C u^{-n}$, for all   $n\geq 0$. 
  
There exists $\delta>0$ and a neighborhood $\mathcal N$ of $J$ 
  such that the restriction to   $\mathcal N$ of the family 
local stable/unstable  manifolds of radius $\delta$ is a lamination, denoted by   $\cW^{u/s}$.  
The Julia set has local product structure so there is a covering by topological bidisks $Q$  (flow boxes) such  that the laminations 
  $\cW^{u/s}$ are trivial in $Q$ and 
   $$J\cap Q \simeq (W^s_Q (x)\cap J  ) \times (W^u_Q(x)\cap J  )  =  (W^s_Q(x)\cap J^-  ) \times (W^u_Q(x)\cap J^+  ).$$ 

It is shown in \cite{bs1} 
that the family of  global stable and unstable  manifolds of points of $J$ also   has  a lamination structure, 
which will be denoted by 
  $\cW^{s/u}$. More precisely, in the dissipative case, $\cW^s$ is a lamination of 
   $J^+$ is laminated by stable manifolds and  
  the other hand, $\cW^u$ is a lamination of  $J^-\setminus \set{a_1, \ldots , a_N}$, where  
  $\set{a_1, \ldots , a_N}$ is the finite set of attracting periodic points of $f$.  
  No unstable leaf extends across an attracting point, even as a singular analytic set: 
  indeed an unstable leaf is biholomorphic to $\C$, therefore  such an extension 
  would yield a submanifold of $\cd$ biholomorphic to a (possibly singular) copy of $\pu$, which is impossible. 

Under additional dissipativity assumptions, it was shown in \cite{lyubich-peters} that the stable lamination $\cW^s$ in $\bb$ can be extended 
to a $C^1$ foliation in some neighborhood of $J^+$: see Lemma~\ref{lem:extension} below. 
  
 Let us conclude this paragraph with a useful elementary result. 
 
 \begin{lem}\label{lem:fatou_disk} 
If $f$ is hyperbolic, every holomorphic disk contained in $K^+$ is either contained in the Fatou set or in the stable manifold of a point of $J$. 
 \end{lem}

\begin{proof}
  Indeed, if $\Delta$ is a disk contained in $K^+$ then  $\Delta$ is a  Fatou disk, i.e. $(f^n\rest{\Delta})_{n\geq 0}$ is a normal family. Now there are two possibilities: either $\Delta$ is contained in $\mathrm{Int}(K^+)$ hence in the Fatou set, or it intersects $J^+$. 
  In the latter case, either $\Delta$ is contained in a stable leaf or 
  by \cite[Lem. 6.4]{bls}, $\Delta$ must have a transversal intersection with some 
  unstable manifold,  so by the Inclination Lemma it is not a Fatou disk, which is a contradiction. 
  \end{proof}

\subsection{Affine structure}\label{subs:affine}
Global stable and  unstable manifolds 
 are uniformized by $\C$, so they admit a natural affine structure. Since any automorphism of $\C$ is affine,  $f$ acts affinely on leaves. 
In particular there is a well defined notion of a round disk, which is $f$-invariant. Likewise,
 the Euclidean distance is well-defined in the leaves, up  to a multiplicative constant.

 For any $x\in J$ we choose a uniformization 
 $\psi^u_x:\C\overset{\sim}{\longrightarrow} W^u(x)$ such that $\psi^u_x(0) = x$ and $\abs{(\psi^u_x)'(0)}=1$ 
  
 \begin{lem}\label{lem:continuity_affine}
 The family of uniformizations  $(\psi^u_x)_{x\in J}$ is continuous up to rotations, that is, 
 if $x_n\to x$ then  $(\psi^u_{x_n})$ is a normal family and its cluster values are of the form 
 $\psi^u_{x}(e^{i\theta} \cdot)$. 
 \end{lem}

\begin{proof} The result follows from the continuity of the affine structure on the unstable leaves (see Theorem~\ref{thm:ghys}).
\end{proof}

It is unclear whether the assignment $J \ni x\mapsto \psi^u_x$ can be chosen to be continuous, that is, if a consistent choice 
of rotation factor $e^{i\theta}$ can be made. This can be done locally but 
 there might be topological obstructions to extend the continuity to $J$. 
  Notice that the $(\psi^u_x)$ provide  a normalization for the leafwise Euclidean distance. The normalized  Euclidean distance 
  on $W^u(x)$ will be denoted by $d_x^u$.. 
  If $C\subset W^u(x)$, its diameter with respect to $d_x^u$ will be denoted by $\diam_x$.
  By Lemma \ref{lem:continuity_affine},
  $d_x^u$ varies continuously with $x$. 
   For $R>0$ we let  $D^u_x(x,R):= \psi^u_x(D(0, R))$. 
   
 By construction, $f$ is a uniformly expanding
  linear map in these affine coordinates, that is $f\circ \psi^u_x = \psi^u_{f(x)}(\lambda^u_x\cdot)$, with 
  $\abs{ \lambda^u_x} = \norm{df\rest{E^u_x}}$. By hyperbolicity there is a positive constant $C$ such
  that for every $x\in J$, 
  \begin{equation}\label{eq:tangent_expansion}
  \abs{\prod_{i=0}^{n-1} \lambda^u_{f^i(x)}}\geq C u^n,
  \end{equation}
where $u>1$ was defined in \S\ref{subs:hyperbolicity}.

By the Koebe Distortion Theorem there exists a uniform $r>0$ such that the   $D^{u}(x,r)$ are contained in the flow boxes (see e.g. 
\cite[Lemma 3.7]{hyperbolic}). By the local bounded geometry of the leaves, the 
distance  induced by the affine structure on the  $D^{u}(x,r)$ is equivalent to that  induced by 
 the ambient Hermitian structure.
Then, iterating finitely many times we can promote this result on the  $D^{u}(x,R)$
 for every given $R>0$.

 All the above discussion holds for stable manifolds, with superscripts $u$ replaced by $s$. 
  
\subsection{Connected and semi-local  components}  For every $x\in J$ (or more generally $x\in K^+\cap \bb$) we denote by $K^+_\bb(x)$ the connected component of $x$ in $K^+\cap \bb$, which is a vertical subset of $\bb$. It follows from the 
  Hénon-like property that $f(K^+_\bb(x))\subset K^+_\bb(f(x))$, thus $f$ induces a (non-invertible) dynamical system on the set of 
  connected components of $K^+\cap \bb$. The same discussion applies to components of $J^+\cap\bb$. 
  More generally, for any   closed  connected subset $C\subset J$ (resp. $C\subset K$),  we define 
  $J^+_\bb(C)$  (resp. $K^+_\bb(C)$)
  to be the connected component of $C$ in $J^+\cap \bb$
  (resp. $K^+\cap \bb$). Of course for $x\in C$,   $J^+_\bb(x)=J^+_\bb(C) $ holds.
   A related concept is $W^s_\bb(x)$, the component of $\bb\cap W^s(x)$ containing $x$. If we set 
   $\displaystyle  W^s_\bb(C)  = \bigcup_{x\in C} W^s_\bb(x)$ then $W^s_\bb(C)$ is contained in $K^+_\bb(C)$ but this inclusion may be strict. This phenomenon may happen  when for some $x\in C$, $ W^s_\bb(x)$  is tangent to $\fr\bb$ (see Figure~\ref{fig:bidisk_stable_component}). 
   
  \begin{figure}[h]
 \includegraphics[width=6cm]{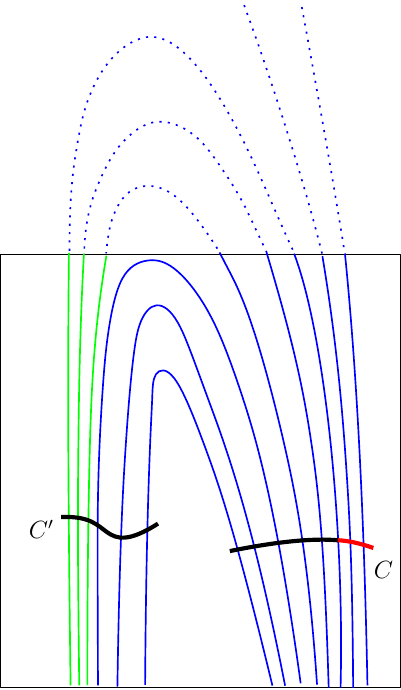}

 \caption{ \begin{small} Discontinued holonomy. The green components belong to $K^+_\bb(C)$ but not to $W^s_\bb(C)$ (in blue). The red part of $C$ cannot be followed under stable holonomy to $C'$ due to a Reeb-like phenomenon. \end{small}} \label{fig:bidisk_stable_component}\label{fig:reeb}
 \end{figure}
 
  For $x\in  K$, we denote by $K^s(x)$ (resp. $K^u(x)$) the connected component of 
  $K\cap W^s(x) = K^-\cap W^s(x)$ 
  (resp. $K\cap W^u(x) = K^+\cap W^u(x)$) containing $x$, and 
  also   $K(x)$ its connected component in $K$.  For $x\in J$, we define $J^s(x)$,  $J^u(x)$ and $J(x)$ similarly. 
  More generally, if needed, we use the notation $\comp_E(x)$ for the connected component of $x$ in  a set $E$. 
  
  We use the subscript `$\mathrm{i}$' to denote topological operations (interior, closure, etc.) relative to the intrinsic topology in 
  stable/unstable manifolds.    

\begin{lem} \label{lem:component_K+J+} 
Assume that $f$ is hyperbolic. Then 
every connected component of $K^+\cap \bb$ has a connected boundary, which is a component of 
$J^+\cap \bb$.
\end{lem}

\begin{proof}
Observe that if $p$ is an interior point of $K^+\cap L$, where $L$ is a horizontal line, then 
it belongs to a Fatou disk. Since $L$ is not contained in $J^+$, by Lemma~\ref{lem:fatou_disk}, 
we get that $p\in \Int(K^+)$. 
This 
implies  that for every  $x\in K^+\cap \bb$, $\fr K^+_\bb(x)  \subset   \bigcup_{t\in \dd}\fr_{L_t}(K^+_\bb(x)\cap L_t)$, where 
$L_t =\dd\times \set{t}$ and $\fr_{L_t}$ refers to the boundary in $L_t$.  The converse inclusion is obvious, so 
$\fr K^+_\bb(x) \cap \bb=   \bigcup_{t\in \dd} \fr_{L_t}(K^+_\bb(x)\cap L_t)$. Since 
$K^+_\bb(x)\cap L_t$ is compact and 
polynomially convex, and  obviously 
$K^+_\bb(x)  =   \bigcup_{t\in \dd}  K^+_\bb(x)\cap L_t$,   this means that 
$K^+_\bb(x) $ is obtained from $\fr K^+_\bb(x)\cap \bb$ by filling the holes of all components of 
$\fr_{L}(K^+_\bb(x)\cap L)$ in every horizontal line. 
 Now  assume $\fr K^+_\bb(x)\cap \bb$ is disconnected, so we can write it as $B_1\cup B_2$, where each $B_i$ is relatively open and $B_1\cap B_2 = \emptyset$. 
In every horizontal slice $L$, $B_i\cap L$ must be a union 
of components of $\fr_{L}(K^+_\bb(x)\cap L)$. For $i=1, 2$, let $\widehat B_i$ be the set obtained by filling the holes of $B_i$ in each horizontal line in $\bb$. The previous discussion shows that 
$K^+_\bb(x) = \widehat B_1\cup \widehat B_2$, where the $\widehat B_i$ are relatively open in $K^+_\bb(x)$ and disjoint. This is a contradiction, therefore  $\fr K^+_\bb(x)\cap \bb$ is  connected. 

For the second statement, simply observe that if $D\subset J^+\cap \bb$ is a connected 
set such that $ \fr K^+_\bb(x)\cap \bb \subset D$, then 
$D$ is contained in $K^+_\bb(x)$ and also in $\fr K^+$ so $D\subset \fr K^+_\bb(x)\cap \bb$ and we are done. 
\end{proof}

   \subsection{Basic properties of leafwise components}
   Here we assume that $f$ is a hyperbolic and dissipative complex Hénon map. 
      The following result is well-known. 
   
   \begin{lem}\label{lem:leafwise_interior}
   For every $x\in K$ we have  
   $\Int_\mathrm{i} (K^u(x))  \subset   \Int(K^+) $ and  $\fr_\mathrm{i} (K^u(x))\subset J$.  In particular  if 
      $\Int_\mathrm{i} (K^u(x))$ is non-empty, each of its components 
  is contained in an attracting basin. Likewise  
  $\Int_\mathrm{i} K^s(x)  = \emptyset$ and $J^s(x)   = K^s(x)$. 
   \end{lem}

\begin{proof}
 Indeed, since stable and unstable manifolds cannot coincide along some open set, 
 if $\Delta$ is a disk contained in $K^u(x)$, 
it follows from Lemma~\ref{lem:fatou_disk} that 
   $\Delta \subset   \Int(K^+)$,  and the remaining conclusions follow.
\end{proof}

   For $x$ in $J$, $K^u(x)$ may be bounded or unbounded 
   for the intrinsic (leafwise) topology. By the maximum principle, $K^u(x)$ is polynomially convex, so if 
   $K^u(x)$ (or equivalently $J^u(x)$) is leafwise bounded, then $K^u(x)$ is simply the polynomially convex hull 
   of $J^u(x)$ (i.e. is obtained by filling in the leafwise bounded components of the complement). 
   
   \begin{lem}\label{lem:leafwise bdd} 
   Given  $x\in K$, in the following properties we have $(iv)\Leftrightarrow (iii) \Rightarrow (ii)\Leftrightarrow (i)$:
 \begin{enumerate}[{(i)}]
 \item    $K^u(x)$ is leafwise bounded; 
 \item $J^u(x)$ is leafwise bounded;
 \item  $W^u_\bb(x)$ is leafwise bounded;
 \item   $W^u_\bb(x)$  is a closed horizontal submanifold of $\bb$. 
   \end{enumerate}
Furthermore if (ii) holds, then (iii) holds for $f^n(x)$ for sufficiently large $n$. 
   \end{lem}

     \begin{proof}  The implication   $( {i})  \Rightarrow ({ii})$ follows directly from the fact that 
     $J^u(x) = \fr_{\mathrm{i}}K^u(x)$. Now assume that  $J^u(x)$ is leafwise bounded. 
Working in $W^u(x)\simeq \C$, we have that  $K^u(x)$ is a closed connected polynomially convex set and 
$J^u(x)$ is a bounded connected component of $\fr_{\mathrm i} K^u(x)$. Since every point of $J^u(x)$ lies 
 on the   boundary of $W^u(x)\setminus K^+$ (for the intrinsic topology), the compact set obtained by filling the holes of $J^u(x)$ must be $K^u(x)$, so the converse implication holds.

     Since $K^u(x)\subset W^u_\bb(x)$,  obviously (\textit{iii}) implies (\textit{i}). Conversely,    
     $K^u(x)$ is the decreasing intersection of the sequence of  components of $x$ in $W^u(x)\cap f^{-n}(\bb)$. Hence,  if
     $K^u(x)$ is leafwise bounded it follows that $\comp_{W^u(x)\cap f^{-n}(\bb)}\lrpar{x}$ is 
     leafwise bounded for large enough $n$, and   so does $W^u(f^n(x))\cap\bb$. 
     
     Recall  that for every $x$, $W^u(x)$ is an injectively immersed copy of $\C$, whose image is a leaf of the lamination of 
      $J^-\setminus \set{a_1, \ldots , a_N}$. Here   the  $a_i$ are the attracting points, 
      and a leaf  never extends to   
      a submanifold in the neighborhood of $a_i$ (\footnote{Indeed otherwise this would induce  a 
      compactification of unstable manifolds, yielding an embedding of $\pu$ into $\cd$.}). 
      In particular, $J^-$ is laminated near $\fr\bb$. If $W^u_\bb(x)$ is 
      leafwise bounded, then  it is of the form $\psi^u_x(\om)$, where $\om$ is some bounded open set in $\C$. Since $\psi^u$ extends to a neighborhood of $\overline \om$,  $W^u_\bb(x)$  it is a properly embedded submanifold of $\bb$, which extends to a neighborhood of $\overline \bb$. So 
     (\textit{iii}) implies (\textit{iv}). Finally, if (\textit{iv}) holds, since $J^-$ is a lamination near $\fr\bb$, we see that $W^u_\bb(x)$ extends to a submanifold $S$ in a neighborhood of $\overline \bb$. 
  Then    $W^u_\bb(x) $ is relatively compact in $ S\subset W^u(x)$ so 
     if $\Omega$ is such that $\psi^u_x(\Omega) = 
      W^u_\bb(x)$ then $\Omega$ is relatively compact in $\C$, and (\textit{iii}) follows.  
               \end{proof}

 \section{External rays}\label{sec:john}
 
 In this section we study external rays along the unstable lamination (i.e. along $J^-$) for a hyperbolic complex Hénon 
 map. The  existence and convergence properties of external rays were studied in the unstably 
 connected case in \cite{bs6, bs7}. Recall that when $\abs{\jac (f)}<1$, unstable  connectedness is equivalent to the connectedness 
 of $J$. The results that we prove here do not rely on any unstable connectivity or dissipativity assumption, 
  nevertheless  what we have in mind is the case of a dissipative unstably disconnected map.

\subsection{Escaping from $K^+$ along an external ray} \label{subs:john}  
By definition,   an \emph{unstable external ray}  (simply called ``external rays'' in the following)
 is  a piecewise smooth continuous
path contained in a leaf $W^u(x)$ 
of the unstable lamination,  which is a union of   gradient lines of $G^+\rest{W^u(x)}$ outside the 
(leafwise locally finite) set
 of critical points of $G^+\rest{W^u(x)}$. As usual we assume 
 that $G^+$ is strictly monotone along external rays 
 (which will be considered as ascending or descending depending on the context).  
We do not prescribe rules for the behavior of rays hitting critical points, so in particular there is no attempt at defining a notion of ``external map''.  

In the next proposition the length of curves is relative to the ambient metric in $\cd$. We
  show that external rays ascend fairly quickly.  

\begin{prop}\label{prop:length}
Let $f$ be a hyperbolic polynomial automorphism of $\cd$ of dynamical degree $d>1$. For every $r_1<r_2$ there exists $\ell(r_1, r_2)$ such that for every $x\in J^-\setminus K^+$  such that if  $G^+(x)= r_1$, any external ray through $x$ reaches  $\set{G^+ = r_2}$ 
along a path whose 
  length is  bounded by $\ell(r_1, r_2)$. 
  In addition   $\ell(r_1, r_2)$ is bounded by a function $\overline\ell(r_2)$ depending only on $r_2$. 
Furthermore  $\ell(r_1, r_2)\to 0$ when $r_1\to r_2$  and $\overline \ell(r_2)  = O(r_2^\alpha)$ when $r_2\to 0$, for some $\alpha>0$.  
\end{prop}

\begin{rmk}\label{rmk:stable_rays}
Notice that no dissipativity is assumed here so the result holds along stable leaves as well. 
\end{rmk}

\begin{proof}
Start with $r_1 = 1$ and $r_2 = d$. In  $J^- \cap \set{1\leq G^+ \leq d}$ the leaves of $\cW^u$ have uniform geometry and 
   no leaf of $\cW^u$ is contained in an equipotential hypersurface  of the form 
   $\set{G^+ = C}$, in particular unstable critical points have uniform order. Thus
    by compactness and continuity of $G^+$, we infer the existence of uniform $\delta_0$ and $\ell_0$ such that 
for every $x\in J^- \cap \set{1\leq G^+ \leq d}$, any external ray through $x$ of length $\ell_0$ reaches 
$\set{G^+ = r}$ with $r\geq G^+(x)+\delta_0$. 
  By concatenating such pieces of rays, we 
deduce the conclusion of the proposition for $r_1 =1 $ and $r_2 = d$ (and $\ell(1, d)\leq (d-1) \ell_0/\delta_0$). Pulling back 
finitely many times and concatenating again, we    get a similar
 conclusion for $ \set{r_0\leq G^+\leq  d}$ for any fixed $r_0$. 
 
  Let us now fix $r_0$ such that $\set{0<G\leq dr_0}\cap J^-$ is contained in $W^u_\loc(J)$. 
  Any piece of external ray between the levels $\set{G^+ = r_0/d^{n}}$ and $\set{G^+ = r_0/d^{n-1}}$ is the pull-back of 
  a piece of external ray in $ \set{ r_0\leq G^+\leq  dr_0}$. Thus by concatenation it follows that any external ray starting from 
  $\set{G^+ = r_0/d^{n}}$  reaches  $\set{G^+ = r_0}$ along a path of 
   length bounded by  $\displaystyle \leq C \ell(r_0, dr_0) \sum_{k=1}^n u^{-k}$, where $u$ is the expansion constant introduced in \S\ref{subs:hyperbolicity}.  
   This proves the existence of the functions $\ell(r_1, r_2)$ and $\overline\ell(r_2)$
   
The same ideas imply immediately that $\ell(r_1, r_2)\to 0$ when $r_1\to r_2$. For the last statement 
 simply note that for every $r_1<r_2\leq r_0$, 
  $$\ell(r_1<r_2) \leq C   \sum_{k=k_0}^\infty u^{-k} = O(u^{-k_0})$$ where $k_0$ is the greatest integer such that $r_0d^{-k_0}\geq r_2$, therefore $\ell(r_1<r_2)  = O(r_2^\alpha)$, with $\alpha = \frac{\log u}{\log d}$. 
 \end{proof}

It is easy to deduce from these ideas that all (descending) 
external rays land. However, since there is no well defined external map, the 
characterization of the set of landing points does not seem to follow directly from this landing property.


 \begin{cor}[John-Hölder property]\label{cor:fast_escaping}  
There exists a constant $\alpha>0$ such that 
 for any sufficiently small  $\eta>0$, for any $x\in J^-\setminus K^+$ sufficiently close to $K^+$,
  there exists a path of length at most $O(\eta^\alpha) $  in $W^u(x)\setminus K^+$ joining $x$ to a point  
  $\eta$-far  from $K^+$. 
\end{cor}

\begin{proof}  
By the previous proposition, there exists a path of length $O(r^{\alpha_1})$ joining $x$ to a point 
$y$ such that $G^+(y) = r$. Now the Green function is Hölder continuous (see \cite{fornaess-sibony}) and 
 that $K^+ = \set{G^+ =0}$, so $d(x, K^+)\geq C r^{\alpha_2}$. The result follows. 
\end{proof}

This  John-Hölder property  has deep consequences for the topology of $K^+\cap W^u (x)$, 
 which will play an important role in the paper.  Intuitively it means that 
 there cannot exist long ``channels'' between local components of $K^+$.

This property is strongly reminiscent of the so-called John   condition for plane domains, 
which have been much studied in one-dimensional dynamics, in relation with non-uniform hyperbolicity 
(see e.g. \cite{CJY, graczyk-smirnov}). 
In the Hénon context, it was shown in \cite{bs7} that for unstably connected hyperbolic maps, 
the components
of $W^u(x)\setminus K^+$ satisfy the John property. It is   very 
  likely that using the continuity of affine structure along unstable leaves, 
  their arguments  can be adapted to the disconnected case as well: this 
   would upgrade Corollary \ref{cor:fast_escaping} to the  actual  John condition.
 One advantage of this weaker property is that it makes no reference to the affine structure of the leaves, 
 so it is more flexible and may be adapted to semi-local situations (e.g. Hénon-like maps).

 \subsection{Accesses and landing} 
%
%
\begin{thm}\label{thm:access}
Let $f$ be a hyperbolic polynomial automorphism of $\cd$ with dynamical degree $d>1$. 
\begin{enumerate}
\item For every $x\in J$,  $D^u(x, 1)\setminus K^+$ admits finitely many connected components, 
and at least one of these components contains $x$ in its closure. 
\item For any component $\Omega$ of $D^u(x, 1)\setminus K^+$ such that $\overline \om \ni x$ there is an external ray   landing at   $x$ through $\om$. 
\end{enumerate}
\end{thm}

For the proof, it is convenient to work in the
affine coordinates given by the unstable parameterizations. We work in the disks $D^u(x,1)$ and measure 
path length    
      relative to the normalized affine metric, which  is equivalent to the ambient one.

\begin{proof}
The first observation is that
 $D^u(x,1)\setminus K^+$ contains  $x$ in its closure: otherwise $x$ would 
lie in the leafwise  interior of $K^+$, thus  contradicting Lemma \ref{lem:leafwise_interior}. Furthermore, 
by the maximum principle, if 
$y\in D^u(x,1)\setminus K^+$ is arbitrary, the component of $y$ in $D^u(x,1)\setminus K^+$    
reaches the boundary of $D^u(x,1)$. 

We claim that  there exists $\eta_1>0$ such that for any $x\in J$ and 
  any component $\om$ of $D^u(x, 1)\setminus K^+$ such that 
$\om\cap D^u\left(x, 1/4\right) \neq \emptyset$, then: $$\sup G^+ \rest {D^u\left(x,1/2\right)\cap \om}\geq \eta_1.$$ This follows directly 
from Proposition \ref{prop:length}: indeed there exists $\eta_1>0$ such that any point of $J^-\setminus K^+$ reaches 
$\set{G^+=\eta_1}$ 
along a path of length $1/4$.
By the Hölder continuity of $G^+$, we infer that any such component $\om$ contains a disk of radius $C\eta_1^\alpha$, so there are 
finitely many of them. 

In particular if $(x_n)$ is a sequence in $D^u(x,1)\setminus K^+$ converging to $x$, 
infinitely many of them must belong to the 
same component $\om$  of $D^u(x,1)\setminus K^+$, which shows that $\overline \om$ contains $x$. This proves  assertion (1) of the theorem.

Fix now a component $\om$ of $D^u(x, 1)\setminus K^+$ such that $\overline \om \ni x$. Let $\eta_1$  
be  as above and fix $\e$ such that $\e<\eta_1/d$  and $\ell(\e, d\e)<\min\lrpar{ 1/2, (u-1)/2}$ where $\ell(\cdot)$ is as in Proposition \ref{prop:length} and the constant $u$ was defined in \S \ref{subs:hyperbolicity}. We do the following construction: for every point $y\in \set{G^+ = \e}\cap \overline 
 D ^u  \lrpar{x, 1/2} $, 
 we consider all ascending external rays emanating from $y$ until they reach $\set{G^+ = d\e}$. The lengths of 
 the corresponding rays is not larger than $ \ell(\e, d\e)$. These are the rays of $0^\mathrm{th}$ generation and we denote by $E_0$ 
 the set of their endpoints (\footnote{Recall that since we do not prescribe the behavior of external rays at critical points of $G^+$ there is no reason that external rays fill up the whole unstable lamination, so $E_0$ could be smaller than $\set{G^+ = d\e}$}), which by the assumption on $\ell(\e, d\e)$ is contained in    $ \set{G^+ = d\e}\cap D^u(x,1)$. We note that 
 $E_0$ is a closed set because it is the ending point set of a compact family of external rays. Since 
 $\e<\eta_1/d$, $E_0$ has non-empty intersection with $\om$. 
  
Performing the same construction in $D^u(f(x), 1)$ we obtain a set of rays of 0th generation  in that disk, 
which connect  $\set{G^+ = \e}\cap \overline 
 D ^u  \lrpar{f(x), 1/2}$ to $\set{G^+=d\e}$, and their endpoints lie in 
 $$\set{G^+ = d\e}\cap \overline D^u \lrpar{f(x), \frac12 + \ell(\e, d\e)}.$$ 
   The pull-backs of these rays by $f$ have their endpoints 
 in  $$\set{G^+ =  \e}\cap \overline D^u \lrpar{x, \unsur{u}\lrpar{\frac12 + \ell(\e, d\e)}}\subset 
 \set{G^+ =  \e}\cap  D^u \lrpar{x, \frac12},$$ by the assumption on $\ell(\e, d\e)$. These are the 
  rays of 1st generation in $D^u(x,1)$.  
   We define $E_1\subset E_0$ to be the closed set of points for which we can concatenate a ray of 0th generation with a ray of 1st generation to descend all the way to $\set{G^+ = \e/d}$. 
    Notice that $f(\om)\cap D^u(f(x), 1)$ is not necessarily connected, so it is 
  a union of components 
 of $D^u(f(x), 1)\setminus K^+$, and since $\overline{f(\om)}\ni f(x)$, at least one
  of these components reaches $D^u(f(x), 1/2)$, so it contains rays of 0th generation.
This shows that  $E_1$ has non-empty intersection with $\om$. 
   
   Continuing inductively this construction, we obtain a decreasing sequence $(E_n)$
   of closed subsets  in $\set{G^+ = d\e}\cap D^u(x,1)$, each of which intersecting $\om$. If $e\in \bigcap_n E_n \cap \om$, 
   then there is a  ray through $e$ (hence  in $\om$) converging to $K^+$, whose part in $\set{\e d^{-n-1} \leq G^+ \leq \e d^{-n}}$ 
   is the pull-back under $f^n$ of a piece of external ray in $D^u(f^n(x), 1)$. Therefore this ray lands at $x$, 
   and the proof  of assertion (2) is 
   complete. 
\end{proof}

\begin{rmk}
The existence of a convergent external ray along any access  to a saddle periodic point  can be obtained 
  exactly as in the 1-dimensional case  (see \cite{eremenko-levin}), without assuming
   uniform hyperbolicity.  
In that case the   Denjoy-Carleman-Ahlfors Theorem is used instead of the John-Hölder property to guarantee
 the finiteness of the number of local 
components. 
\end{rmk}

\subsection{Topology of $K^+\cap W^u$}\label{subs:topology}
In this section we review the  consequences of  Corollary \ref{cor:fast_escaping} 
  for  the topology of unstable components of 
$K^+$. 

\begin{thm}\label{thm:topology}
Let $f$ be a hyperbolic Hénon map. Then for every $x\in J$:
\begin{enumerate}[(i)]
\item every component of $K^+\cap W^u(x)$ (resp. $J^+\cap W^u(x)$) 
is locally connected;
\item for any smoothly bounded domain   $\om\subset W^u(x)$, for every $\delta>0$,
$K^+\cap \om$ (resp. $J^+\cap \om$) 
admits at most finitely many components of diameter larger than $\delta$.  
\end{enumerate}
\end{thm}

As before this follows from    \cite{bs7}
when $f$ is unstably connected (see Theorems 3.5 and 5.6 there), so we focus on the unstably disconnected case. 
In this case it is known that 
$K^+\cap W^u(x)$ has uncountably many point components (see \cite[Thm 3.1]{bs7}). Using (\emph{ii}) we can be more precise:

\begin{cor}\label{cor:points}
Let $f$ be   hyperbolic and unstably  disconnected. Then for every $x\in J$, all but at most countably many components of 
$K^+\cap W^u(x)$ are points.
\end{cor}

Let us stress that the conclusions of the theorem  follow solely from Corollary \ref{cor:fast_escaping} together with 
  some elementary topological considerations. 
   Remark also that the assumption that $\om$ has smooth boundary   in (\textit{ii}) is necessary:
indeed otherwise  it could cut a component of $K^+$ in infinitely many parts of large diameter (think e.g. 
 of the closed unit square cut out by some comb-like domain).

Part or all of Theorem~\ref{thm:topology}
 is  presumably  known to specialists, however for completeness we provide some details.
Let us first define a   notion of ``fast escaping   from a compact set''.  

\begin{defi}
Let $\om$ be a smoothly bounded domain in $\C$ and 
$K$ be a closed subset in $\om\subset \C$. We say that $K$ satisfies the 
\emph{fast escaping property} in $\Omega$ 
if 
there exists an increasing continuous 
 function $\ell$ with $\ell(0) = 0$ such that 
 for any sufficiently small
 $\eta>0$ and any $x\notin K$, there exists a path $\gamma:[0, 1]\to \om \setminus K$ of length  at most $\ell(\eta)$ such that 
 $\gamma(0) = x$ and $d(\gamma(1), K)\geq \eta$.  
\end{defi}

Corollary \ref{cor:fast_escaping} asserts that if $f$ is hyperbolic, then for every $x\in J$, and 
 any leafwise bounded domain   $\om\subset W^u(x)$, 
$K^+\cap W^u(x)$ 
satisfies the fast escaping  property in $\om$ with $\ell(\eta) =  c\eta^\alpha$. Note that both properties (\textit{i}) and (\textit{ii}) 
in Theorem \ref{thm:topology} are local in $W^u(x)$ so the choice of ambient or leafwise topology or metric is harmless. 

The following lemma takes care of  item (\textit{ii}) of the theorem. 

\begin{lem}\label{lem:diameter}
Let $K$ be a closed subset of a smoothly bounded domain $\om\subset \C$, 
satisfying the fast escaping property. 
Then for every $\delta>0$, there are at most finitely many components of $K$ 
(resp. of $\Int(K)$, of  $\fr K$)
of diameter  greater than $\delta$. 
\end{lem}

\begin{proof}
We first prove the result for $K$ and $\Int(K)$ 
and then explain how to modify the proof to deal with $\fr K$.  
Let us first assume that $\om$ is the unit square $Q$, and denote by  $\pi_1$ and $\pi_2$ the coordinate projections of $Q$. 
Assume by contradiction that there are  infinitely many components $(C_i)_{i\geq 0}$ of $K$ with diameter $\geq \delta$. 
Then there exists  $\pi\in \set{\pi_1, \pi_2}$  such that infinitely many $C_i$ 
satisfy $\diam(\pi(C_i))\geq \delta/2$. Therefore there is an interval $I$ of length $\delta/4$   such that for infinitely 
many $i$,  $C_i$ disconnects the strip $\pi\inv(I)$, and we conclude that 
$ \pi\inv(I) \setminus \bigcup C_i$ has infinitely many connected components $U_j$ going all the way 
across the strip. (Notice that the $U_j$ may contain other points of $K$.)
Let $c$ be the center point of $I$. Since the $C_i$ are distinct components of $K$, for each $j$ there exists 
a point  $x_j$ in $U_j\cap \pi^{-1}(c)$ which does not belong to $K$. 
 If $\eta$ is chosen such that $\ell(\eta) = \delta/20$ we infer from the fast escaping property that for every $j$, 
 $U_j$ contains a disk of radius $\eta$, which is the desired contradiction. 
 
 For $\Int(K)$ the argument is identical except that instead of $c$ we take a small open interval 
 $I'$ about  $c$ and argue that if the $C_i$ are distinct components of $\Int(K)$, there 
 exists $x_j\in U_j\cap \pi^{-1}(I')$ which does not belong to $K$.
 
 In the general case, take a square $Q$ such that $\om \Subset Q$ and replace $K$ by $K' = \overline {K\cap \om}$. Let us check 
 that  $K'$ satisfies the fast escaping property in $Q$. Indeed, if $x\in Q\setminus K'$ we have either $x\in \om$,   $x\in \fr\om$ or $x\in Q\setminus \overline \om$. In the first case we take the  path $\gamma$ given by the fast escaping property of $K$ in $\om$. 
 In the second case, any small ball  $B$ about  $x$  intersects 
 $\om\setminus K$, and we simply take a path starting from some $x' \in B\cap (\om\setminus K)$. Finally in the last case we use the 
 fact that  $\overline \om$ has the fast escaping property in $Q$. 
 
 By the first part of the proof we conclude that $K'$ has finitely many components of diameter $\geq \delta$. Since any component
 of $K$ (resp. $\Int(K)$)  is contained in a component of $K'$ (resp. $\Int(K')$), we are done. 
 
  The proof that $\fr K$ admits only finitely many components of diameter greater than $\delta$ goes exactly along the same lines. We assume that there are infinitely many components $C_i$ of 
 $\fr K $ disconnecting the strip $\pi\inv (I)$, so that 
 $ \pi\inv(I)  \setminus \bigcup C_i$ also has infinitely many components $U_j$. 
 The   difference with the previous case is  that  some of these components may be completely included in $K$. 
 We modify the argument as follows.   Denote by $U'_j$ the components completely included in $K$ and by $U''_j$ the remaining ones. We claim that there are infinitely many $U''_j$'s. Indeed since the $C_i$ are  components of $\fr K$, two components of the form $U'_j$ must be separated by a component of the form $U''_j$. So there are infinitely many such components. Then we take a small open interval $I'\subset I$ containing $c$ and we  repeat this argument, to obtain that there are infinitely many $j$'s 
 such that $U''_j\cap \pi\inv(I')$ contains a point $x_j$ that does not belong to $K$. Then we proceed with the proof as in the previous case, by constructing infinitely many disjoint disks of radius $\eta$ in $Q$ to get a 
 contradiction. 
\end{proof}

\begin{proof}[Proof of    {(i)} in   Theorem~\ref{thm:topology}] 
Since $J^+\cap W^u(x) = \fr_{\rm i}( K^+\cap W^u(x))$, general topology implies  that 
local connectivity of $J^+\cap W^u(x) $ implies 
that of $K^+\cap W^u(x)$ (see \cite[\S 49.III]{kuratowski-vol2})
so it is enough to focus on $J^+$. For convenience we plug in some dynamical information. 
Since $f$ is unstably disconnected, it admits an 
unstable transversal $\Delta^u$,
 that is a horizontal disk of finite degree in 
$\bb$ contained in some unstable manifold (of a periodic saddle point, say). For every $x\in J$, $W^s(x)$ intersects $\Delta^u$: this 
easily follows from the density of $W^s(x)$ in $J^+$ and the local product structure. Fix $y\in W^s(x)\cap \Delta^u$. By using the 
   local   holonomy along the stable lamination
 $W^u_\loc(x) \to W^u_\loc(y)$, we see that 
 $J^+\cap W^u(x)$ is locally connected at $x$ if and only if 
 $J^+\cap W^u(y)$ is locally connected at $y$. Therefore it is enough to show that 
 $J^+\cap \Delta^u$ is locally connected. Since $K^+\cap \Delta^u$ is  polynomially convex and compactly contained 
 in $\Delta^u$  it follows that 
 $\om:=\Delta^u\setminus K^+$ is connected and $ J^+\cap \Delta^u = \fr \om$.  
 Likewise every component of $\fr\om$ is of the form $\fr A$, where $A$ is a component of $\Delta^u\cap K^+$. For 
 such a component, by Carathéodory's Theorem local connectivity of $\fr A$ is equivalent to that of $A$, 
 which is of course equivalent 
 to local connectivity of $A$ at every point of its boundary. Let us 
    fix $x_0\in \fr A$: to complete the proof we have to  show that $A$ is locally connected at $x_0$.

Assume by contradiction that $A$ is not locally connected at $x_0$. Then for 
 small $\e>0$ such that if $C$ denotes the 
component of $A\cap \overline B(x_0, \e)$, then $x_0  = \lim x_n$, where   $x_n$ belongs to  $A  \setminus C$. Without loss of 
generality we can assume that $x_n\in B(x_0, \e/2)$. Let 
$C_n = \comp_{A\cap \overline B(x_0, \e)}(x_n)$, which by definition is disjoint from $C$. Passing to a subsequence if necessary, we may assume that the $C_n$ are disjoint (the construction here is similar to that of convergence continua in 
  \cite[\S 49.VI]{kuratowski-vol2}). 
Since $C$ and the $C_n$ intersect $\fr B (x_0, \e)$, their diameter is bounded from below by some $\delta>0$.  
From this point the proof is similar to that of of Lemma \ref{lem:diameter}: we can find an orthogonal 
  projection $\pi$ such that $C$ and the $C_n$ cross the strip $\pi\inv(I)$ horizontally and  $\pi\inv(I)\setminus (C\cup \bigcup C_n)$ 
  admits infinitely many connected components $U_j$ going all the way across the strip. If $\pi\inv(c)$ denotes the center line of the  
  strip, for every $j$, $\pi\inv(c)\cap U_j$ has non-trivial intersection with $\om$, and the fast escaping property of $\om$ 
gives a contradiction as before.   
\end{proof}

 \subsection{Complement: John-Hölder property in basins}  
We illustrate the   comments from \S~\ref{subs:john} 
on the 
versatility of the John-Hölder property by sketching   a proof of the following result. 

\begin{thm}\label{thm:john-holder_basin}
Let $f$ be a hyperbolic polynomial automorphism of $\cd$, and $\mathcal B$ be an attracting  basin. Then the John-Hölder property holds in $\mathcal B$, i.e. for any component $\Omega$ of 
$\mathcal B\cap W^u(x)$ there exists $\eta_0$ depending only on  $\Omega$ such that  
for any $y\in   \Omega$ sufficiently close to $J$, there exists a path in 
of length $O(\eta^\alpha)$ in $W^u(x)$ joining $y$ to a point $\eta$-far from $J$. 
\end{thm}

\begin{rmk}\label{rmk:eta0}
A difference between this result and Corollary~\ref{cor:fast_escaping} is that in  
Corollary~\ref{cor:fast_escaping} the constant $\eta_0$ is independent of the component of $W^u(x)\setminus K^+$, because $G^+$ reaches arbitrary large values  in each component. 
Here the situation is different because  
$\mathcal B\cap W^u(x)$ typically has (infinitely) 
 many small components, so how far we can get from the boundary 
 really depends on the component. 
\end{rmk}

\begin{proof} 
For convenience we present a proof which is purposely close to that 
of Proposition~\ref{prop:length} and Corollary~\ref{cor:fast_escaping}. Replace $f$ 
some iterate so that $\mathcal B$ is the basin of attraction of a fixed point $a$ with multipliers 
$\lambda_1, \lambda_2$, with $\abs{\lambda_2}\leq  \abs{\lambda_1}$. 
 There exists
 a biholomorphism   $\phi: \mathcal B \to \C^2$, which conjugates the dynamics to that of the triangular map 
 $(z_1, z_2)\mapsto (\lambda_1 z_1 + r(z_2), \lambda_2 z_2)$, where $r$ is a polynomial 
 which is non-zero only when there is a resonance  
 $\lambda_2\neq \lambda_1^j$ between the eigenvalues (see~\cite{sternberg}). Introduce the  function 
 $$\tilde H(z_1, z_2) = \abs{z_1 - r(z_2/\lambda_2)}^{2} +  \abs{z_2}^{2\alpha}, \text{ where } 
 \alpha  = \frac{\log\lambda_1}{\log\lambda_2}\geq 1$$ and put $H = \tilde H \circ \phi$. 
 This is a smooth strictly psh function on $\mathcal B$ which satisfies $H\circ f = \abs{\lambda_1}^2 H$.
To  get a better analogy with the 
 previous case we may consider  $H\inv$ which satisfies 
 $H\inv\circ f = \abs{\lambda_1}^{-2} H\inv$, and tends to zero when approaching $J$. 
The  restriction of this function 
to any local unstable disk in $\mathcal B\setminus\set{a}$ 
 is non-constant and one easily checks that its   set of critical points is discrete.  
 
 Arguing in Proposition~\ref{prop:length}, we define a family of  
 rays in $\mathcal B$ by considering gradient lines of $H $ (or equivalently $H\inv$) 
 along $\cW^u$ , first 
 in the  fundamental domain $\set{\abs{\lambda_1}^2\leq H\inv  \leq 1}$ 
 and then in $\set{0<H\inv  \leq 1}$ by pulling back. 
It follows that for every component $\Omega$ of $\mathcal B\cap W^u(x)$, 
 for every $0<r_1<r_2< \max_\Omega\abs{H\inv}$, and any $y\in \Omega$ such that 
 $H\inv(y)  = r_1$, there exists a ray of length $\ell(r_1, r_2) = O(r_2^\alpha)$ 
 joining $y$ to a point of $\set{H\inv = r_2}$. 
 
To conclude the argument we need to adapt the proof of Corollary~\ref{cor:fast_escaping},
which relies on the Hölder continuity of the Green function. Instead we use an argument based 
on uniform hyperbolicity.  Indeed, let $x\in J$ and $y\in W^u_\loc(x)$ be such that 
$d^u(x,y) = \e$. We want to show that $H\inv(y)\lesssim \e^\alpha$ for some $\alpha$. 
By the expansion along unstable manifolds and the local uniform geometry 
 it takes at most $N\leq C \abs{\log\e}$ 
iterates to map $y$ into a given compact subset of $\mathcal B$. Hence 
$$H\inv(y)   = \abs{\lambda_1}^{2N} H\inv(f^N(y)) \leq C \abs{\lambda_1}^{2N} 
\leq  C \abs{\lambda_1}^{2C \abs{\log\e}} = C\e^{-2C\log\abs{\lambda_1}}$$ and we are done. 
\end{proof}

 \section{Stable total disconnectedness}\label{sec:total disconnected}
 
We say that $f$ (or $J$)
is {\em  stably totally disconnected} if for every $x\in J$, $W^s(x)\cap J^-$ is totally disconnected. 
Note that since $J$ has local product structure with respect to the stable and unstable laminations, 
  $W^s(x) \cap J = W^s(x) \cap J^-$.

 \begin{prop}\label{prop:stable total disc}
 Let $f$ be a hyperbolic Hénon map. The following assertions  are equivalent.
 \begin{enumerate}[(i)]
 \item Every leaf of the stable lamination in $\bb$ is a vertical  submanifold of finite  degree. 
 \item The leaves of the stable lamination in $\bb$ are vertical  submanifolds of uniformly bounded degree.
 \item For every $x$ in $J$, $J^s(x) = K^s(x) = \set{x}$, that is, $f$ is stably totally disconnected. 
\end{enumerate} 
 \end{prop}
 
 Note that dissipativity is not required here, so this result holds in the unstable direction as well.

 \begin{proof}
 The implication  $(ii)\Rightarrow(i)$ is obvious 
  and its converse  $(i)\Rightarrow(ii)$  follows from the   semi-continuity properties of the degree and 
 is identical to  \cite[Lemma 5.1]{lyubich-peters}. 
 To prove that $(iii)\Rightarrow(i)$ we use 
 Lemma~\ref{lem:leafwise bdd}   for the   stable lamination: indeed if 
  $J^s(x)$ is a point for every $x$, then all 4 conditions of 
   Lemma~\ref{lem:leafwise bdd} 
  are equivalent, and the equivalence of properties \textit{(ii)} and \textit{(iii)} there yield the result. 
 Finally, 
 $(ii)\Rightarrow(iii)$ does not require hyperbolicity and 
 was established in \cite[Prop. 2.14]{connex}. For convenience, let us  recall the argument: 
  for every vertical disk $D$ of degree  $\leq k$, and every component $D'$ of $D\cap f(\bb)$, the modulus of the annulus 
  $D\setminus D'$ is bounded below by $m = m(k)>0$, 
 and for every $x\in J$ there is an infinite nest of such annuli surrounding 
  the component of $x$ in $W^s(x) \cap J$. So $W^s(x) \cap J$ is totally disconnected and we are done. 
    \end{proof}

A   way to ensure the boundedness of the   degrees of semi-local 
 stable manifolds   originates in \cite{tangencies} and relies on Wiman's theorem for  entire functions. 
The following result is contained in \cite{lyubich-peters}. 

\begin{prop}
Let $f$ be a hyperbolic Hénon map such that $\abs{\jac f}\leq d^{-2}$. Then $f$ is stably totally disconnected. 
\end{prop}

\begin{proof}[Proof (sketch)]
Fix $x\in J$ and $v\in E^s(x)$. Uniform  hyperbolicity together with  the assumption on the Jacobian imply
that $\norm{df^n_x(v)} \leq C s^n$, where $s< d^{-2}$. Denote  as before $\psi^s_\bullet$ the normalized stable parameterization. 
It follows  that $f^n\circ \psi^s_x (\cdot)= \psi^s_{f^n(x)} (\lambda_n\cdot)$, where $\abs{\lambda_n}\leq  C s^n$. 
Then from the relation $$G^-\circ \psi^s_x(\lambda_n\inv\zeta) = d^n G^-\circ \psi^s_{f^n(x)}(\zeta)$$
 we deduce that  $G^-\circ\psi^s_x$ is  a subharmonic function 
of order smaller than $1/2$ and Wiman's theorem implies that $\comp_{(\psi^s_x)\inv(\bb)}(x)$ is a bounded domain in $\C$, thus 
$W^s_\bb(x)$ has bounded vertical degree and we are done.  
\end{proof}

 Another  idea, which  was communicated to us by Pierre Berger,
  is to use  a Hausdorff dimension argument to prove directly  that stable slices 
 of $J$ are totally disconnected. Indeed the Hausdorff dimension of stable slices of $J^-$ can be estimated using    
 thermodynamic formalism for  hyperbolic maps. This   turns out to give  a better bound on the Jacobian. 
 
 \begin{prop}\label{prop:berger}
Let $f$ be a hyperbolic Hénon map such that $\abs{\jac f}< d^{-1}$. Then $f$ is stably totally disconnected. 
\end{prop}

\begin{proof}
Since $J$ is a locally maximal hyperbolic set and the dynamics along   stable manifolds is conformal,  
there is an exact formula  for the Hausdorff dimension  of $J\cap W^s_{\loc(x)}$ for any $x\in J^-$, given by:
 \begin{equation}\label{eq:thermo}
\delta^s:=  \dim_H\lrpar{ J\cap W^s_{\loc}(x)} = \frac{h_{\kappa^s}(f)}{ -\int \log \abs{df\rest{E^s(x)}} d \kappa^s(x) }  
 \end{equation} 
  (see Pesin's book \cite[Thm 22.1]{pesin_book}; this goes back to the work of Manning and McCluskey \cite{mccluskey-manning}), 
  where $\kappa^s$ is a certain invariant measure (the unique equilibrium state associated to $\delta^s\log \abs{df\rest{E^s}}$) and 
  $h_{\kappa^s}(f)$  is its measure theoretic entropy.  By the variational principle we have that 
   $h_{\kappa^s}(f)\leq \log d$. On the other hand 
    the Lyapunov exponent in  the denominator in the right hand side of 
     \eqref{eq:thermo} is bounded below by $\abs{\log \abs{\jac f}} > \log d$. Therefore 
    we conclude that $ \dim_H\lrpar{ J\cap W^s_{\loc}(x)} <1$ from which 
     it follows that $J\cap W^s_{\loc}(x)$ is totally disconnected.  
\end{proof}

\begin{question}
Is a dissipative hyperbolic Hénon map always stably totally disconnected? 
\end{question}

 \section{Classification of semi-local components of $K^+$ and $J^+$}\label{sec:semi_local}
 
\emph{Throughout this section, $f$ is a  dissipative and  hyperbolic complex Hénon map of degree $d$ with a disconnected Julia 
 set (or equivalently, $f$ is unstably disconnected). }
  We assume   moreover that 
 \emph{$f$ is stably totally disconnected.} The results 
 of \S \ref{sec:total disconnected} imply that  this holds whenever $\abs{\jac f}<1/d$. 
 We fix a large bidisk $\bb$ as before, and our purpose is to classify the connected components of 
 $J^+\cap \bb$  
 and study the 
 induced dynamics on this set of components.

 \subsection{Geometric preparations} We start with some general lemmas about vertical submanifolds in a bidisk. 
 We define  the  angle $\angle (v,w)$ 
 between two complex directions $v$ and $w$ at $x\in \cd$  to be  their distance 
in $ \pp(T_x\cd)\simeq \pu$ relative   to the Fubini-Study  metric induced by the standard Hermitian structure of 
 $T_x\cd\simeq \cd$. 

\begin{lem}\label{lem:quasi_tang}
Let $M$ be a vertical submanifold in $\dd\times \dd$, 
and  let  $a\in \dd$ and $r>0$ such that 
$M$ has no horizontal tangency in $\dd\times D(a,2r) $. Then there exists a universal  constant 
$C_0$  
such that for any $x\in   \dd\times D(a,r) $, the angle between $T_xM$ and the horizontal direction is bounded from 
below by $C_0r$.
\end{lem}

\begin{proof}
If $M$ has no horizontal tangency in $\dd\times D(a,2r) $, then
$M\cap (\dd\times D(a,2r) )$ is the union of $\deg(M)$ vertical graphs. Let $\Gamma$ be one of these graphs. Then 
$\varphi:=\pi_1\circ (\pi_2\rest\Gamma)\inv $ maps $D(a,2r)$ into $2 \dd$  and 
$\Gamma = \set{(\varphi(w),  w), \ w\in D(a,2r)}$. By the Cauchy estimate, we get that 
$\abs{\varphi'}\leq 2/r$ on $D(a,r)$ and the result follows. 
\end{proof}
 
A typical  use of this result is by taking the contraposite:  if a vertical submanifold $M$  in $\dd\times \dd$   
has a near horizontal tangency in $\dd\times D(a,r) $, then 
 it has an actual horizontal tangency in $\dd\times D(a,2r) $. Let us  denote by  $[e_1] \in \pp( T\cd)$ the horizontal direction. 
 
 \begin{cor}\label{cor:near_tangency}
Let $M$ be a vertical submanifold in $\dd\times \dd$  which extends as a vertical submanifold 
to $\dd\times (3/2)\dd$. There exists a universal constant $C_1$ such that if 
  for some 
 $a\in \dd$, there exists $x\in M\cap (\dd\times \set{a})$ such that $\angle(T_xM, [e_1]) < \theta$, then there exists $a'\in (3/2)\dd$ 
 such that $\abs{a-a'}< C_1\theta$ and $M$ is tangent to $ \dd\times \set{a'}$. 
 \end{cor}

 For the sake of completeness let us also  state 
 a slightly stronger result:
   
\begin{cor}\label{cor:near_tangency_precised}
Let $M$ be a vertical submanifold in $\dd\times \dd$ of degree at most 
$k$ which extends as a vertical submanifold 
to $\dd\times r_0\dd$ for some $r_0>1$ (say $r_0  =3/2$). There exists a function $h=h_k$ such that 
$h(\theta)\to 0$ as $\theta\to 0$ with the following property: 
if   $x\in M$ is such that the angle between 
$T_xM$ and the horizontal direction is bounded by $\theta\ll 1$ then there exists $x'\in M$ with 
$d(x,x')\leq h (\theta) $ such that  $M$ has a horizontal tangency at $x'$. 
\end{cor}

 \begin{proof}
 Indeed, letting $a = \pi_2(x)$, and applying 
  Corollary \ref{cor:near_tangency}  we see that  the connected component of 
  $M$  containing $x$ in $D(a,  C_1\theta)\times \dd$ cannot be a vertical graph, so it admits a horizontal tangency.  
  Furthermore, an easy compactness argument shows that the diameter of   a connected component  of
  $M\cap D(a, r)\times \dd$ is bounded by $h_k(r)$ with $h_k(r)\to 0$ as $r\to 0$. The result follows. 
   \end{proof}

\begin{rmk}
It is likely that $h_k(r) =  O\lrpar{r^{1/k}}$  but   the precise argument needs 
 to be found. 
\end{rmk}

The following result is a  precise  version of the Reeb stability theorem
(see \cite{candel-conlon}) which is specialized to our setting.

 \begin{lem}\label{lem:reeb}
 Let 
  $x_0\in J$ be such that 
$W^s_{  \bb}(x_0)$   is transverse to $\fr  \bb$. 
Then there exists $\delta$ depending only on 
$\min_{y\in {W^s_{\overline{\bb}}(x_0)}\cap \fr\bb} \angle 
\lrpar{T_ y W^s_{\overline{\bb}}(x_0), [e_1]}$ such that if 
$\tau\subset J^u(x_0)$ 
is a connected compact set containing $x_0$, of 
diameter less than $\delta$, then for every $x\in \tau$, $W^s(x)$ is transverse to $\fr\bb$, 
$\deg W^s_{\bb}(x) = \deg W^s_{r\bb}(x_0)$ and 
$\bigcup_{x\in \tau} W^s_{\bb}(x)$ is homeomorphic to $\tau\times W^s_{\bb}(x_0)$.  
 \end{lem}

Note that it is slightly abusing to say that $W^s_{ \bb}(x)$   is transverse to $\fr (\bb)$ since $W^s_{   \bb}(x)$ precisely stops at $\fr \bb$. Of course $W^s_{   \bb}(x)$ extends to a neighborhood of $\overline{ \bb}$ and what we mean is transversality for 
 this extension.  

\begin{rmk}\label{rmk:uniformity_r}
Later on we will use this lemma with $r\bb$ instead of $\bb$ for $1\leq r\leq 2$ (see Proposition~\ref{prop:thin_connected_component}). 
It will be important there that the constant $\delta$ is uniform with 
$r\in [1,2]$, which easily follows from the proof. 
\end{rmk}

\begin{proof}
Set $\theta = \min_{y\in W^s_{\overline\bb}(x_0)\cap \fr\bb} \angle \lrpar{T_ y W^s_{\overline\bb}(x_0), [e_1]}$.
 The stable lamination in a neighborhood of $\overline {\bb}$ is covered by finitely many flow boxes. 
 So there exists $r>1$ depending 
 only on $\theta$ such that $W^s_{r\bb}(x_0)$ is transverse to $\fr(r\bb)$. Since the stable leaves in $\bb$ are simply connected, 
 we can apply a  local version of the Reeb stability theorem 
 (see \cite[Prop. 11.4.8]{candel-conlon}) which asserts that
 when  $\tau\subset  J\cap W^u(x_0)$ is sufficiently small,  
 for $x\in \tau$, by local triviality of the stable lamination, 
  the domain $W^s_{r\bb}(x_0)\subset W^s(x_0)$ 
 can be lifted 
  to a domain   $D_x\subset W^s(x)$, and the collection  $\set{D_x, \  x\in \tau}$ 
 is topologically a product. Since $W^s_\bb(x_0)$ is transverse to $\fr\bb$, $W^s_\bb(x_0)\subset W^s_{r\bb}(x_0)$  is a smoothly 
 bounded domain and,  reducing  $\tau$
  if necessary, the transversality persists, 
   $\comp_{D_x \cap \bb}(x)$ varies continuously and 
  $\bigcup_{x\in \tau} W^s_\bb(x)$ is a product. Finally, if we fix any horizontal line, say close to $\fr\bb$  by transversality and 
  continuity,  its number of   intersection points with $W^s_\bb(x)$ is constant, hence the statement on the degree. 
  
  What remains to be seen is why the size of the allowed 
   transversal $\tau$ depends only on the minimal angle $\theta$. This follows from the 
mechanism of Reeb stability. What we need to know is how far we can 
push $x$ in $\tau$  so as to keep the transversality between $W^s_\bb(x)$ and $\fr\bb$. 
Pick $y\in W^s_{\overline\bb}(x_0)\cap \fr\bb$. Understanding how a neighborhood of $y$ in $W^s(x_0)$ 
evolves  when the base point $x\in \tau$ changes depends on the choice of a path $\gamma$ 
joining $x_0$ to $y$ in $W^s(x_0)$ and of a covering of $\gamma$ by a chain of overlapping  plaques. (Recall that by definition 
a \emph{plaque} is the intersection between a leaf an a flow box.)
 Notice first that there is a uniform control of the length of a 
  such a path $\gamma$: for instance  we can take an external ray and apply Proposition \ref{prop:length} (see Remark~\ref{rmk:stable_rays}). 
  So the length of a minimal chain of plaques joining $x_0$ to $y$ is uniformly bounded, and there exists 
  $\delta = \delta (\theta)$ such that if $\diam_{x_0} (\tau)<\delta$, then the continuation of the plaque containing $y$ 
  remains transverse to $\fr\bb$.  
  Finally, the  number of plaques required to cover $\fr W^s_{\overline\bb}(x_0)$ depends basically on the volume of 
  $W^s_{r\bb}(x_0)$ for some $r>1$, which in turn depends only on the degree of $W^s_{r'\bb}(x_0)$ for some $r'>r$.  
  By Proposition \ref{prop:stable total disc} this degree is uniformly bounded. So the number of plaques is uniformly 
  bounded and we are done.  
\end{proof}

We will also need the following extension lemma. 

\begin{lem}[{\cite[Prop. 5.8]{lyubich-peters}}]\label{lem:extension}
There exists a neighborhood $\mathcal N$ of $J^+\cap \bb$ such that the 
 stable lamination $\cW^s$ 
 extends to a  $C^1$ foliation of $\mathcal N$. 
\end{lem}

Observe that in~\cite{lyubich-peters} it is assumed that
  $\abs{\jac f}< d^{-2}$ but what is really needed for extending the stable lamination is the boundedness of the vertical degree which 
  holds in our setting (cf. Proposition~\ref{prop:stable total disc}). 
   The  $C^1$ regularity of the holonomy will not be used in the paper. 

Using this extension lemma, we can extend Lemma \ref{lem:reeb} to a statement about an open neighborhood
of $W^s_\bb(x_0)$ with exactly the same proof. 

   \begin{lem}\label{lem:reeb_extended}
 Let 
  $x_0\in J$ be such that 
$W^s_{  \bb}(x_0)$   is transverse to $\fr \bb$. 
Then there exists $\delta$ depending only on 
$\min_{y\in {W^s_{\overline{r\bb}}(x_0)}\cap \fr\bb} \angle \lrpar{T_ y W^s_{\overline{\bb}}(x_0), [e_1]}$ 
such that for every $x\in D^u(x_0, \delta)$, $\cW^s(x)$ is transverse to $\fr\bb$, 
$\deg \cW^s_{\bb}(x) = \deg W^s_{\bb}(x_0)$ and 
$\bigcup_{x\in D^u_{x_0}(x_0, \delta} W^s_{\bb}(x)$ is homeomorphic to $D^u_{x_0}(x_0, \delta)\times  W^s_\bb(x_0)$.  
 \end{lem}

 \subsection{Thin and thick components}\label{subs:small_components} 
In this section we study the geometry of the components of $J^+\cap \bb$. 
The arguments rely mostly on 
 the geometry of the stable lamination, not on  the dynamics of $f$.
One main result is  that thin components of $K^+\cap\bb$ have a simple leaf 
 structure (Proposition~\ref{prop:thin_connected_component}). 
 It follows that for a given component of $J^+\cap\bb$, either all its 
  unstable slices are small, or all of them are large 
 (Proposition~\ref{prop:alternative}). Together with the results of 
 \S\ref{subs:topology} this   leads to  a description and some regularity 
 properties of    components of $J^+\cap\bb$ and $K^+\cap\bb$. 
 
  We start with a simple case. 

\begin{prop}\label{prop:s-holonomy_point} If $x \in J$ is such that $K^u(x) = J^u(x) = \set{x}$ then $K^+_\bb(x) = J^+_\bb (x) = W^s_\bb(x)$. 
\end{prop}

 \begin{proof} As observed above the inclusion $W^s_\bb(x) \subset K^+_\bb(x)$ is obvious. For the converse inclusion, 
 observe that 
 for every $n\in \Z$, $K^u(f^n(x)) = \set{f^n(x)}$. For $n\geq 1$, consider a small loop 
 $\gamma_n\subset W^u(f^n(x))$ around $f^n(x)$ that is disjoint 
 from $K^+$. By the local product structure we can  extend it to a germ of 3-manifold $\widetilde \gamma_n$
  transverse to $W^u(f^n(x))$, disjoint from $K^+$, and of  size uniformly bounded from below in the stable direction. 
  Since $W^s_{2\bb}(x)$ has finite vertical degree in $2\bb$, it admits 
   finitely many horizontal tangencies, so we can fix $1\leq r\leq 2$ such that 
  $W^s_{r\bb}$ is transverse to $\fr(r\bb)$.  Then 
  by the Inclination Lemma, 
  for large $n$,  $f^{-n}\lrpar{\widetilde \gamma_n}$ contains a small 
  ``tube'' around $W^s_{r\bb}(x)$  whose boundary is disjoint  from $K^+$. 
  It   follows that $K^+_{r\bb}(x)  =  W^s_{r\bb}(x)$,  hence $K^+_\bb(x) \subset W^s_{r\bb}(x)\cap \bb$. 
  Finally $W^s_{r\bb}(x)\cap \bb$  has finitely many components, and one of them is 
  $W^s_\bb(x)$, so $K^+_\bb(x) = W^s_\bb(x)$. 
\end{proof}

Here is a first interesting consequence. 
  
  \begin{cor}\label{cor:vertical_submanifolds}
  All but countably many components of $K^+\cap \bb$ are vertical submanifolds. 
  \end{cor}
 
 \begin{proof}
 Fix a global unstable transversal $\Delta^u$ in $\bb$. Then every component of $K^+\cap \bb$ intersects $\Delta^u$.  
 Indeed, for any such component $C$,  $\fr C$ is contained in  $J^+$ so it contains stable manifolds. 
 Stable manifolds in $\bb$ are vertical and of finite degree, so they have non-trivial (transverse) intersection with $\Delta^u$. 
Now if  $C$ is  non-trivial, that is, not reduced to a vertical submanifold, then by Proposition \ref{prop:s-holonomy_point}, 
any component of $C\cap \Delta^u$ is non-trivial, and the result follows from Corollary  \ref{cor:points}. 
 \end{proof}

  Another case where $J^+_\bb(x)$ is easily understood is when stable leaves are transverse to $\fr\bb$. 
 
 \begin{prop}\label{prop:transverse}
 Assume that    $J^u(x)$ is a leafwise bounded component such that for every 
 $y\in J^u(x)$, $W^s_\bb(y)$ is transverse to $\fr\bb$. Then 
 \begin{equation}
 \label{eq:J+B}
 J^+_\bb(x)  = \bigcup_{y\in J^u(x)} W^s_{\bb}(y).
 \end{equation}
 \end{prop}

Note that  this result is   not true if the transversality assumption  is omitted 
(see Figure \ref{fig:bidisk_stable_component} 
for a visual explanation). 

\begin{proof}
Let $C$ be defined by the right hand side of~\eqref{eq:J+B}.  Since  the $W^s_{\bb}(y)$, $y\in J^u(x)$, are 
transverse to $\fr\bb$, they vary continuously with $y$. It follows that $C$ is a closed connected set. 
To show that $C= J^+_\bb(x) $, it is convenient to use the extension of the stable lamination to a neighborhood of $J^+\cap \bb$ (given in Lemma~\ref{lem:extension}).   Let $(U_n)$ be 
a basis of open neighborhoods of $J^u(x)$ in $W^u(x)$  such that for every $n$,
  $\fr U_n \cap J = \emptyset$. For every 
  $\delta>0$, $U_n$ is contained in the $\delta$-neighborhood of $J^u(x)$ for large $n$.
   Thus,  by Lemma \ref{lem:reeb_extended}  
  the leaves issued from $U_n$  are transverse to $\fr \bb$ and stay close to  $C$. 
  Let $\widetilde U_n$ be the saturation of $U_n$ 
  in the extended foliation. Then  $(\widetilde U_n)$ is a basis of neighborhoods of $C$ 
  in $\bb$ such that 
  $\fr \widetilde U_n$ is disjoint from $J^+$. We conclude that $C= J^+_\bb(x)$.  
\end{proof}

The structure of $J^+_\bb(x)$ is not so easy to describe without this transversality assumption. 
  Still, the argument can (almost) be salvaged if $J^u(x)$ is small enough. This will be a key property  in the following. 

\begin{prop}\label{prop:thin_connected_component} 
There exists $\delta_1>0$ such that if    $x \in J$ is such that 
 $\diam_x(J^u(x))\leq \delta_1$, then  there exists $1\leq r\leq 2$ such that 
 for every $y\in J^u(x)$, $W^s_{r\bb}(y)$ is transverse to $\fr(r\bb)$ and 
 $J^u(x)$ can be followed under holonomy along $W^s_{r\bb}(x)$. In particular 
 $J^+_{r\bb}(x)$ is homeomorphic to $J^u(x)\times W^s_{r\bb}(x)$ and 
 \begin{equation}\label{eq:J+2B}
J^+_\bb(x) \subset J^+_{r\bb}(x)  =    W^s_{r\bb}(J^u(x)) \subset  W^s_{2\bb}(J^u(x)) = \bigcup_{y\in J^u(x)} W^s_{2\bb}(y).
\end{equation}
\end{prop} 

Recall that   $\diam_x$ denotes the diameter relative to the normalized 
leafwise metric $d_x^u$ induced by the affine structure.  
By polynomial convexity, if $K^u(x)$ is leafwise bounded, then 
$J^u(x) = \fr_{\mathrm i} K^u(x)$ so 
$\diam_x (K^u(x)) = \diam_x (J^u(x))$. Recall from \S\ref{subs:affine} that 
 by the Koebe Distortion Theorem, the ambient distance $d$ and the leafwise Euclidean  
distance  $d_x^u$ are equivalent in a small neighborhood
of $x$, with universal bounds, i.e. in some neighborhood of  $x$ in $W^u(x)$ we have $ d/2\leq d_x^u\leq 2 d$.  
In particular if 
  $\diam_x(J^u(x))$ is small enough  then $\diam (J^u(x))$ and $\diam(K^u(x))$ are comparable to 
      $\diam_x(J^u(x))$ (where $\diam$ denotes the ambient diameter).

\begin{proof}[Proof of Proposition \ref{prop:thin_connected_component}] 
Recall that every leaf of the stable lamination in $3\bb$ is a vertical disk of degree bounded by $D$, 
so by the Riemann-H\"urwitz formula it admits at most  $D-1$ horizontal tangencies. 
For $k=0, \ldots , D$, let $r_k = 1+ \frac{k}{D}$, and fix 
  $\theta < \frac{C_0}{8D}$, where $C_0$ is as in Lemma  \ref{lem:quasi_tang}. Let $x\in J$ be arbitrary. 
  By the pigeonhole principle, 
  there exists $k\in \set{0, \ldots , D-1}$ such that $W^s_{2\bb}(x)$ has no horizontal tangency in 
  $r_{k+1} \bb\setminus r_k \bb$. So by Lemma~\ref{lem:quasi_tang}
   (scaled to $2\bb$ and applied to any $a$ such that  $\abs{a} = R (r_k+r_{k+1})/2$, where 
  $R$ is the radius of $\bb$) we infer that  
  $$\min_{y\in  \cap \fr( r_k'\bb)} \angle (T_ y W^s_\bb(x_0), [e_1]) \geq\theta, \text{ where } r_k'= \frac{r_k+ r_{k+1}}{2}.$$
Therefore, by Lemma~\ref{lem:reeb} and Remark~\ref{rmk:uniformity_r}
 there exists $\delta_1$ 
depending  only on $\theta$,
hence ultimately only on $D$, hence on $f$, 
 such that  if $\diam_x(J^u(x))\leq \delta_1$,  
then for every $y\in J^u(x)$, $W^s_{r_k'\bb}$ is transverse to $\fr (r_k'\bb)$ and 
$W^s_{r_k'\bb}(J^u(x))$ is topologically a product.  This completes the proof of the first 
part of the proposition. From this point, the description of $J^+_{2\bb}(x)$ in~\eqref{eq:J+2B} directly  follows  from Proposition~\ref{prop:transverse}. 
\end{proof}

It follows from this analysis  that if $C$ is a semi-local component of $J^+$, then either 
all its unstable slices are large or all of them are small. 

\begin{prop}\label{prop:alternative}
There exists $0<\delta_1\leq \delta_2$ such that  for every   component $C$ of $J^+\cap \bb$ 
the following alternative holds:
\begin{enumerate}[(i)]
\item either for every $x\in C\cap J$, $\diam_x J^u(x)\leq \delta_2$;
\item or for every $x\in C\cap J$, $\diam_x J^u(x)> \delta_1$. 
\end{enumerate}
In addition if (i) holds then $C$ satisfies the conclusions of Proposition \ref{prop:thin_connected_component}.
\end{prop}

Referring to this dichotomy in the following, 
we will say that a component is    \emph{thin} (resp.  \emph{thick}) if it satisfies 
  (\emph{i})  (resp.  (\emph{ii})). We stress that the Proposition asserts that a component is thick as soon as \emph{one}
  of its unstable slices  has intrinsic diameter larger than $\delta_2$. 
  As seen before (see e.g. 
Corollary \ref{cor:vertical_submanifolds}), if $\Delta^u$ is an unstable transversal, 
 every  semi-local component of $J^+$ intersects  $\Delta^u$, so from Theorem \ref{thm:topology} we immediately deduce:
 
 \begin{cor}\label{cor:thick_components}
 There are only finitely many thick  components of $J^+\cap \bb$. 
 \end{cor}

  Proposition~\ref{prop:alternative}  is a direct consequence of the following lemma.

\begin{lem} \label{lem:s-holonomy_component} 
Let $\delta_1$ be as in Proposition \ref{prop:thin_connected_component}. 
There exists $\delta_2\geq \delta_1$ such that 
if $x$ is such that $\diam_x (J^u(x))\leq \delta_1$, then for every $y\in  J^+_\bb(x)\cap J$,  $\diam_y(J^u(y))\leq \delta_2$. 
\end{lem}

\begin{proof}
Indeed by Proposition~\ref{prop:thin_connected_component},
 if $\diam_x (J^u(x))\leq \delta_1$, then any   point in $J^+_\bb(x)$ can be joined to   $y\in J^u(x)$
by a path contained in $W^s_{2\bb}(y)$. Furthermore, as explained in the 
proof of Lemma~\ref{lem:reeb}, 
the plaque-length of such a $\gamma$ is uniformly bounded.  
The bound on $\diam_y(J^u(y))$   then follows
 from the uniform continuity of holonomy along bounded paths in the stable lamination. 
 \end{proof}


\begin{rmk}\label{rmk:general_thick_thin}
The argument of Propositions~\ref{prop:thin_connected_component} and~\ref{prop:alternative} makes no use of the fact that $J^u(x)$ is a component of $J\cap W^u(x)$. Thus the same statements hold for the saturation by semi-local stable leaves of any (say closed) subset $X$ of an unstable manifold: 
if its diameter of $X$ is small enough then, changing the bidisk $\bb$ if necessary, 
the  saturation $\hat X$ of $X$ by semi-local stable manifolds 
is a product and all the stable slices of $\hat X$  have a small diameter. 
\end{rmk}

\begin{prop}\label{prop:finiteness}
Let $\Delta^u$ be an unstable transversal in $\bb$. 
For every connected component $C$ of 
$J^+\cap \bb$ (resp. $K^+\cap \bb$), $C\cap \Delta^u$  admits finitely many  connected components. 
\end{prop}

\begin{proof}
Let us first discuss the case of components of $J^+\cap \bb$. For thick components, the result follows immediately from 
Corollary \ref{cor:thick_components}, so we may assume that $C$ is thin. 
As already seen,  $C$ intersects $\Delta^u$.
Pick $x\in C\cap \Delta^u$, in particular $x\in J$. 
Since $C$ is thin,  for some $1\leq r\leq 2$, $W^s_{r\bb}(x)$ 
is transverse to $\fr(r\bb)$ and by Proposition \ref{prop:thin_connected_component}, $J^u(x)$ can 
be followed under holonomy along $W^s_{r\bb}(x)$. Since  $W^s_{r\bb}(x)$   and  $\Delta^u$ have finitely 
many intersection points, we infer  that $J^+_{r\bb}(x)\cap \Delta^u$ has finitely many connected components. Finally, 
$J^+_\bb(x) = C$ coincides with 
 the component of  $J^+_{r\bb}(x)\cap \bb$ containing $x$, so 
$C\cap \Delta^u$ is a union of connected components of $J^+_{r\bb}(x)\cap \Delta^u$  and we conclude that 
  there are finitely many of them.

We now discuss components of $K^+\cap \bb$. 
Recall from Lemma~\ref{lem:component_K+J+} that for such  
a component $C$,  $\fr C$ is a   component  of 
 $J^+\cap \bb$. 
  Assume first that 
  such a component $A$ is thin. Given  $x\in A\cap \Delta^u$. 
 $J^u(x)$ can be followed under holonomy along $W^s_{r\bb}(x)$ for some $1\leq r\leq 2$. 
If the polynomial hull  of $J^u(x)$ is non-empty, 
 then it has a small diameter and it can be followed by holonomy in $r\bb$ 
 along the extended foliation just as in Proposition \ref{prop:thin_connected_component} and it is topologically a product.  
 It follows that $C\cap \Delta^u$ is the polynomial hull of $J^+_\bb(x)\cap \Delta^u$ and it has finitely many components. 
 On the other hand, 
 if every component of $\fr C$ is thick, then $\fr C\cap \Delta^u$ is contained in the finitely 
 many components of $K^+\cap \Delta^u$ of 
 diameter greater than some $\delta$, and so is $C\cap \Delta^u$. 
 This concludes the proof. 
\end{proof}

We conclude this subsection by  giving 
 a general description of    components of $J^+\cap \bb$. Fix an unstable transversal 
$\Delta^u$. Let $x\in J\cap \Delta^u$ and consider  $W^s_{\bb}(J^u(x))  = \bigcup_{y\in  J^u(x)} W^s_\bb(y)$. If every $W^s_\bb(y)$ is 
transverse to $\fr\bb$ then by Proposition \ref{prop:transverse}, $W^s_{\bb}(J^u(x)) = J^+_\bb(x)$. In the general case 
we define a relation between components of $J^+\cap \Delta^u$ by declaring that $C_1\leftrightarrow C_2$  if and only if 
there exists $x\in C_1$ such that $W^s_\bb(x)\cap C_2 \neq \emptyset$ (or equivalently there exists 
$(x_1, x_2)\in C_1\times C_2$ such that $W^s_\bb(x_1) = W^s_\bb(x_2)$). 
Then extend this relation to an equivalence relation 
(still denoted by $\leftrightarrow$) by allowing finite chains $C_1, \ldots , C_n$. Finally we define
$$\widehat W^s_{\bb}(J^u(x))  := \bigcup_{C \leftrightarrow J^u(x)} \, \bigcup_{y\in  C} W^s_\bb(y).$$

\begin{prop}\label{prop:chapeau}
For any $x\in J$, $J^+_\bb(x)$ coincides with $\widehat W^s_{\bb}(J^u(x))$. 
\end{prop}

\begin{proof}
By Proposition \ref{prop:finiteness}, $J^+_\bb(x)\cap \Delta^u$ 
admits finitely many connected components $(C_i)_{i\in I}$. 
 Every point   $z\in J^+_{\bb}(x)$ 
belongs to some $W^s_\bb(y)$, $y\in \Delta^u$, and necessarily $y$ belongs to some 
$C_i$, say $C_{i_0}$. 
 Furthermore, if $z'\in J^+_{\bb}(x)$ is close to $z$, by the continuity of stable manifolds, there exists 
$y'\in \Delta^u$ close to $y$ such that $z'\in W^s_\bb(y')$. Since the $C_i$ are at positive distance from 
each other it follows that $y'$ belongs to $C_{i_0}$. 
In other words,  $W^s_\bb(C_{i_0})$ is relatively open in $J^+_\bb(x)$. Clearly  $W^s_\bb(C_{i_0})$ is 
connected, and even arcwise connected since by Theorem \ref{thm:topology} $C_i$ is locally connected. 
Thus the $W^s_\bb(C_i)$ realize a  finite
  cover of $J^+_\bb(x)$ by connected open sets, which are
contained in or  disjoint from $J^+_\bb(x)$. Define a non-oriented 
  graph on $I$ by joining $i$ and $j$ whenever $W^s_\bb(C_i)\cap W^s_\bb(C_j)\neq \emptyset$. If we fix $i_0$ such that 
$W^s_\bb(C_{i_0})\subset J^+_\bb(x)$, it follows that $J^+_\bb(x) = \bigcup_{i\in I_0} W^s_\bb(C_i)$ where 
$I_0$ is the component of $i_0$ in the graph. This is exactly the announced description. 
\end{proof}

Let us point out the following interesting consequence of the proof:

\begin{cor}
Every connected component of $J^+\cap \bb$ (resp. $K^+\cap \bb$) is locally connected. 
\end{cor}

\begin{proof}
Given a component $J^+_\bb(x)$ of $J^+\cap \bb$,   
with notation as in the previous proof, $(W^s_\bb(C_i))_{i\in I}$ is a finite cover of  $J^+_\bb(x)$  by 
locally connected and relatively open sets: local connectedness follows. If now $C$ is a component of $K^+\cap \bb$, 
we saw in the proof of Proposition \ref{prop:finiteness} that $\fr C$ is a finite union of components of $J^+\cap \bb$, 
therefore  $\fr C$ is locally connected. General topology then implies  
that $C$ is locally  connected and we are done.
\end{proof}

 \subsection{Induced dynamics on  the set of components of $J^+$} \label{subs:dynamics_components}
 We still consider a uniformly hyperbolic dissipative Hénon map, with a disconnected and stably totally  disconnected Julia set,
 and fix  a large bidisk $\bb$ as before.
 Since $f$ maps $K^+\cap \bb$ (resp. $J^+\cap \bb$)    into itself, it induces a dynamical system on the set of its connected
 components. Recall that a component is said \emph{non-trivial} if it is not reduced to a vertical submanifold.

  \begin{thm}\label{thm:component}
  Let $f$ be   dissipative and hyperbolic with a disconnected and stably totally disconnected Julia set
 and $\bb\subset\cd$ be a large bidisk. Then $K^+\cap \bb$ (resp. $J^+\cap \bb$)  admits 
 uncountably many components, at most countably many of which being non-trivial.
  Any non-trivial connected component of $K^+\cap \bb$ (resp. $J^+\cap \bb$) 
 is preperiodic,  and there are finitely many non-trivial periodic components. 
   \end{thm}

\begin{rmk}
\label{rmk:trivial_components}
Notice a periodic component of $K^+\cap \bb$ can  be trivial, that is,
 a vertical submanifold.  Since 
   it is mapped into itself by some $f^N$ in this case we conclude that it is of the form $W^s_\bb(x)$ for 
   some saddle periodic point $x$. 
\end{rmk}

\begin{lem}\label{lem:usc_diameter}
The function $y\mapsto \diam_y(J^u(y))$ (resp. $y\mapsto \diam_y(K^u(y))$) is   upper semi-continuous on $J$. 
In particular if $y_n\to y_\infty$ and $(\diam_{y_n}(K^u(y_n)))$ is unbounded, then $K^u(y_\infty)$ is leafwise unbounded, 
and likewise for $J^u$. 
\end{lem}

\begin{proof}
Recall that 
 $\diam_y(J^u(y)) = \diam_y(K^u(y))$  for every $y\in J$  (including the case where it is infinite)
 so it is enough to deal with $K^u(y)$.
Assume first that  the $y_n$ belong to the same local leaf and $y_n\to y_\infty$. 
If $K^u(y_\infty)$ is leafwise bounded, we can consider a closed loop $\gamma$ enclosing it and disjoint from $K^+$. Then for large enough $n$, $ \gamma $ also encloses 
$K^u(y_n)$, and any cluster value  of this sequence for the Hausdorff topology is a continuum contained in $K^+$ and containing $y_\infty$. It follows that $$\limsup_{n\to \infty} \diam_{y_\infty}  ( K^u(y_n))\leq  \diam_{y_\infty}  ( K^u(y_\infty))  $$ hence 
$$\limsup_{n\to \infty} \diam_{y_n}   (K^u(y_n))\leq  \diam_{y_\infty}  (K^u(y_\infty)), $$ as desired.  Of course 
if  $K^u(y_\infty)$ is leafwise unbounded, the inequality is obvious.

Assume now that  the  $y_n$ belong to different  local leaves. As 
before, the case where $K^u(y_\infty)$ is leafwise unbounded  is obvious. If 
$K^u(y_\infty)$ is leafwise bounded, again we consider 
a closed loop $\gamma$ enclosing it and disjoint from $K^+$. In addition we can assume
that $\diam_{y_\infty}(\gamma)$ is arbitrary close to $\diam_{y_\infty}(K^u(y_\infty))$. 
When $y_n\to y_\infty$, $\gamma$  can be lifted to a loop $\widetilde \gamma_n$
in $W^u(y_n)$, with roughly the same diameter (here we use the continuity 
of the leafwise distance $d^u_y$), and $K^u(y_n)$ is enclosed in $\widetilde \gamma_n$. 
The semi-continuity of the diameter follows. 
\end{proof}

\begin{proof}[Proof of Theorem \ref{thm:component}]
Fix an unstable transversal $\Delta^u$, and recall that any component of $K^+\cap \bb$ (resp. $J^+\cap \bb$) 
 intersects $\Delta^u$. 
By \cite[Thm 7.1]{bs6}, $J^+\cap \Delta^u$ 
 admits uncountably many point components, thus the first assertion of the theorem follows from Proposition 
 \ref{prop:s-holonomy_point}. Then Corollary \ref{cor:vertical_submanifolds} asserts that at most countably many 
 components are non-trivial. 
 
 Let $x\in  J^+\cap \Delta^u$ and assume that $J^+_\bb(x)$ (or equivalently $K^+_\bb(x)$) is non-trivial. Since $\Delta^u$ is a 
 global transversal, $J^u(x)$ is leafwise bounded.  For $n\geq 0$,   $J^u(x_n)   = f^n(J^u(x))$  where 
 $x_n = f^n(x)$,  and by \eqref{eq:tangent_expansion},   
 \begin{equation}\label{eq:diameter}
 \diam_{x_n} (J^u(x_n)  ) 
 \geq C u^n \diam_x(J^u(x))\underset{n\to\infty}{\longrightarrow}\infty.
 \end{equation}
  Let $x_\infty$ be any accumulation point of $(x_n)$. 
 By Lemma \ref{lem:usc_diameter}, $J^u(x_\infty)$ is leafwise unbounded, and so does  $K^u(x_\infty)$.
 
 By local product structure, for large $n$, the   holonomy along the stable lamination defines  a 
 projection   
 $$D^u(x_n, 3/2)\cap J^+ \to  D^u(x_\infty, 2)\cap J^+$$
 which we simply denote by $\pi^s$. It is Lipschitz (see Lemma~\ref{lem:extension}) and  a homeomorphism onto its image. 
 Notice that $\pi^s (D^u(x_n, 3/2)\cap J^+)$ contains $D^u(x_\infty, 1)\cap J^+$ for large $n$.  
  For large $n$, $J^u(x_n)$ intersects the boundary of  
 $D^u(x_n, 3/2)$, so the sets  $J^u(\pi^s(x_n))$ define a sequence of components of $J^+\cap D^u(x_\infty, 1)$ of diameter bounded from below. From Theorem \ref{thm:topology} 
 we infer that this sequence is finite. Let us denote by 
 $C_j$, $j=1, \ldots, N$  these components.  By the Pigeonhole Principle there exist 
  $n\neq n'$ such that 
 $\pi_s(x_n)$ and  $\pi_s(x_{n'})$ belong to the same $C_j$, thus $x_n$ and $x_{n'}$ belong the local stable saturation of $C_j$. 
 Therefore  the sequence $(J^+_\bb(x_n))$   is eventually periodic, and so is $(K^+_\bb(x_n))$.   

Consider now a non-trivial  periodic component $C$ of  $J^+\cap \bb$. 
Then  it is of the form 
$J^+_\bb(x)$ for some $x\in \Delta^u\cap J^+$. The previous argument shows that there are points $x'\in C\cap J$ such that 
$J^u(x')$ is leafwise unbounded. By Proposition \ref{prop:alternative},   the components of the 
 slices $J^+_\bb \cap \Delta^u$ have  
diameter uniformly bounded from below (here we use the fact that for every $x\in \Delta^u\cap J^+$, 
the distance 
$d_x^u$ is uniformly comparable to the ambient distance on $\Delta^u$).  Thus, by 
 Theorem \ref{thm:topology} only
 finitely many such components can arise and we conclude that  $C$ belongs to a finite set of components. The corresponding 
 result for components of $K^+\cap\bb$ follows from Lemma~\ref{lem:component_K+J+}. \end{proof}
  
\begin{rmk}   
Using techniques similar to those of  \S \ref{subs:small_components}
it is easily seen  that  any component of $K^+\cap \bb$ has finitely many preimages. In other words, the 
induced dynamical system on components of $K^+\cap \bb$ is finite-to-1.
Indeed assume by contradiction that 
  $C$ is a component such that  $f\inv(C)\cap \bb$ has infinitely many preimages $C_i$. 
 Then by Theorem \ref{thm:topology}, 
 for some $i$,  $C^i\cap\Delta^u$ has a component of   small diameter. 
  Therefore by pushing forward, there is some $x\in C\cap J$ such that $\diam_x (J^u(x))$ is small, that is, 
  $J^+_\bb(x)$  (or equivalently $K^+_\bb(x)$)
  is thin. But it is easy to show that a thin component admits finitely many preimages, and we arrive at the desired contradiction. \qed
    \end{rmk}

  \section{Components of $J$ and $K$}\label{sec:components}
  
 We keep the same setting as before, that is, $f$ is a 
 uniformly hyperbolic dissipative Hénon map, with a disconnected and stably totally  disconnected Julia set. In this section, we complete the proof of the main theorem by classifying the connected 
 components of $J$ and $K$.

We start with an easy fact. Recall the notation $E(x) = \comp_E(x)$.

 \begin{prop}\label{prop:cpt_leafwise_bdd}
 If $x\in J$ is such that $J^u(x)$ is leafwise bounded then $J(x)   = J^u(x)$.  
 \end{prop}

\begin{proof}
First, $  {J^u(x)}$ is a connected set such that $x\in  {J^u(x)} \subset J$ so it is contained in $J(x)$. 
To prove  the converse statement, 
let $(U_n)$ be a  sequence  of open neighborhoods of $J^u(x)$ in $W^u(x)$ decreasing to $J^u(x)$ and such that 
$\fr_{\mathrm{i}} U_n \cap J = \emptyset$. 
Since $J^s(x)   = \set{x}$, for every $n$  any sufficiently small loop $\gamma$ about $x$ in $W^s(x)$ can be propagated 
along $U_n$ to yield an open set $\widetilde U_n$ such that $\fr\widetilde U_n = \emptyset$. Note that we did not prove any 
extension result for the unstable lamination, so we cannot simply say that we propagate $\gamma$ by using some 
``unstable holonomy''. On the other hand we can simply use the inclination lemma, by 
pushing forward a small  thickening of $f^{-n}(\gamma)$ as a 3 manifold transverse to $W^s(f^{-n}(x))$. 
Finally,   for every $n$, $\fr\widetilde U_n $ is relatively open and closed in $J$, so it contains $J(x)$ and we conclude 
that $J(x)= J^u(x)$.
\end{proof}

To understand the structure of periodic components of $J$, let us introduce a definition. 

      \begin{defi}\label{defi:quasi_solenoid}
 A   \emph{quasi-solenoid}    is a   saddle 
 hyperbolic set such that $f^k(\Lambda) = \Lambda$ for some $k$
 and: 
 \begin{itemize}
 \item $\Lambda$ is connected;
 \item $\Lambda$ has local product structure;
 \item  for every $x\in \Lambda$, $\Lambda\cap W^u(x)$ is leafwise unbounded and 
 locally  connected,
and  $\Lambda\cap W^s(x)$ is totally disconnected.  
\end{itemize}
  \end{defi}

Observe that in this definition we do not require 
 that $\Lambda\cap W^s_\loc(x)$ is a Cantor set.
 In other words, we allow for isolated points in a stable transversal 
 (this phenomenon  will be ruled out later under 
 appropriate hypotheses, see Theorem~\ref{thm:complement_mixing}).

\begin{thm}\label{thm:invariant_component}
  Let $f$ be   dissipative and hyperbolic with a disconnected and stably totally disconnected Julia set
 and $\bb$ be as above. Let $C$ be a periodic component of $J^+\cap \bb$ and $k$ be its period.  
 Then $\Lambda :=  \bigcap_{n\geq 0} f^{kn}(C)$ is  a point or a  quasi-solenoid, 
 and it is a connected component of $J$.
  \end{thm}

  \begin{proof} 
  Replacing $f$ by some iterate, we may  assume $C$ is invariant, that is, $k=1$.
  If $C$ is a vertical manifold, it follows from 
   Remark \ref{rmk:trivial_components} that $\Lambda$ is a point, and the other properties follow easily, 
   so the interesting case is when 
  $C$ is non-trivial. Then, arguing  in the proof of Theorem \ref{thm:component}, 
  by~\eqref{eq:diameter}, $C$ contains points such that $\diam_x(J^u(x))$ is arbitrary large, so it is 
thick in the sense of Proposition \ref{prop:alternative}. 
  Define $\Lambda  := \bigcap_{n\geq 0} f^n(C) = \bigcap_{n\geq 0} \overline{f^n(C)} $. 
  Since by assumption $f(C)\subset C$, 
    $\Lambda$ is a decreasing intersection of compact connected sets. 
  Hence $\Lambda$ is an invariant connected hyperbolic  set contained in $J$, and 
  $f(\Lambda ) = \Lambda$. Let us show that it is a connected component of $J$. For this, let $\Lambda'$ be the 
  connected component of $\Lambda$ in $J$. By definition $\Lambda\subset \Lambda'$. 
 Since $\Lambda'$ is connected and    contained in $J^+\cap \bb$, it must be contained in $C$. Furthermore since $f(\Lambda ) = \Lambda$,  and $f$ permutes the components of $J$, we have that 
 $f(\Lambda' ) = \Lambda'$, hence 
  for every $n\geq 1$, $f^{-n}(\Lambda')\subset C$, and we conclude that $\Lambda'\subset \bigcap_{n\geq 0} f^n(C) = \Lambda$, as was to be shown. 
  
We claim that for every $x\in \Lambda$, $J^u(x)$ is leafwise unbounded. 
    Indeed for every $x\in \Lambda$,  we have that $x= f^n(x_{-n})$ with $x_{-n}  = f^{-n}(x) \in C$ 
    and since $C$ is thick,   
    $\diam_{x_{-n}}(J^u(x_{-n}))$ is uniformly bounded from below, and the result follows. 
 
    By Lemma~\ref{lem:diameter}, for every $x\in \Lambda$, there are only finitely many components of 
    $J\cap D^u(x,1)$ intersecting  $\fr_{\mathrm{i}}D^u(x,1)$ and $D^u(x,1/2)$. 
    A simple compactness argument using the holonomy invariance of $J^+$
    shows that this number is uniformly bounded, therefore  
there exists a uniform $\delta>0$ such that leafwise unbounded components of   $J^+$ 
    intersecting $D^u(x,1/2)$
    are $\delta$-separated in $D^u(x,1)$  relative to the distance $d_x^u$ (or equivalently, relative to 
     the ambient one). 
   From this we deduce that for every $x\in \Lambda$, there exists $\delta>0$ such that 
     $\Lambda$ coincides with $J^u(x)$ in $W^u_\delta(x)$, and it follows from Theorem~\ref{thm:topology}
     that $\Lambda$ 
     is locally connected in the unstable direction. 
     
  Let us   show that   $\Lambda$ has local product structure.  For this, let  $y_1, y_2\in \Lambda$ be 
    close (i.e. $d(y_1, y_2)\ll \delta$), denote by $\pi^s: W^u_\loc(y_1)\to W^u_\loc(y_2) $ 
      the projection along stable leaves, and let  $z_2  =\pi^s(y_1)$. 
  Since $J^u(y_1)$ and $J^u(y_2)$ are  leafwise unbounded, 
  if $d(y_1, y_2)$ is small enough, $J^u(z_2)$ intersects $\fr_{\rm i} D^u(y_2, 1)$, and 
  so does $J^u(y_2)$. By definition of $\delta$, it follows that $J^u(y_2) = J^u(z_2)$, hence 
 $y_2$ and $z_2$ belong to the same connected component of $J$. In particular, 
    $z_2$ belongs to $C$. Since $f\inv$ contracts distances along unstable manifolds, and respects connected components of $J$,  we can repeat this argument with 
  $f^{-n}(y_2)$ and $f^{-n}(z_2)$ for any $n\geq 0$ and we conclude that  
  $z_2\in \Lambda$, as was to be shown.   
  \end{proof}

\begin{thm}\label{thm:component_J}  
Let $f$ be   dissipative and hyperbolic with a disconnected and stably totally disconnected Julia set. Then every 
component of $J$ is either 
\begin{enumerate}
\item   a point;
\item  or of the form $J^u(x)$ with $J^u(x)$ non-trivial and leafwise bounded; 
\item or a periodic  quasi-solenoid. 
\end{enumerate}

In addition: 
\begin{enumerate}[(i)]
\item There are finitely many quasi-solenoidal components
\item Every   periodic component of $J$ is either a point or a  quasi-solenoid. 
\item Every non-trivial component of $J$ is attracted by a quasi-solenoid.  
More precisely, given a non-trivial component $C$  for every $\delta>0$ there 
exists $n$ such that $f^{kn}(C)\subset W^s_\delta(\Lambda)$, where 
$\Lambda$ is a quasi-solenoid of period $k$. 
\end{enumerate}
\end{thm}

Note that in assertion~(\emph{ii}), the uniformity of $n$ as a  function of $\delta$ is not a direct consequence of 
the fact that $\omega(C) \subset \Lambda$. 

\begin{proof}
To establish the announced trichotomy, by Proposition \ref{prop:cpt_leafwise_bdd} 
it is enough to  show that if $C$ is a component such that for some 
$x\in C$, $J^u(x)$ is leafwise unbounded, then $C$ is a periodic 
quasi-solenoid. Note that for every $n\geq 1$,  $J^u(f^{-n}(x))$
  is leafwise unbounded. 
Therefore the component of $f^{-n}(x)$ in $J^+\cap\bb$ is thick in the sense of 
Proposition~\ref{prop:alternative}, and by 
Corollary~\ref{cor:thick_components}, $J^+_\bb(f^{-n}(x))$ belongs to a finite set of semi-local components. 
Thus there exists a component  $C^+$  
  of $J^+\cap \bb$ and an infinite sequence $n_i$ such that $f^{-n_i}(x)\in C^+$, hence 
$C^+$ is periodic of some period $k$ and reversing time we get that 
$J^u(x)$ is included in $\Lambda :=\bigcap_{n\geq 0} f^{kn}(C^+)$. 
By  Theorem~\ref{thm:invariant_component}, $\Lambda$ is a  
  quasi-solenoid and $J(x) = C = \Lambda$. 
  
  Since 
there are only finitely many periodic semi-local components of $J^+$, this argument shows that 
$J$ has only finitely many solenoidal components.

For assertion~(\emph{ii}), let   $C$ be a  
 periodic component of $J$ which is not reduced to a point, and let $x\in C$.  
 Without loss of generality we assume $C$ is fixed. 
 Expansion in the unstable direction shows that if   $J^u(x)$ is leafwise bounded, 
 then $J^u(x) = \set{x}$, which is a contradiction. Thus by the first part of the proof, 
 $C$ is a quasi-solenoid. 
 
To prove (\emph{iii}), let $C$ be a non-trivial component of $J$, and 
for some large bidisk $\bb$, let $C^+$ be the component of $J^+\cap \bb$ containing $C$. Then by Theorem \ref{thm:component}  $C^+$  is ultimately periodic (with preperiod $k$), thus by 
Theorem \ref{thm:invariant_component}, $\bigcap_{n\geq 0} f^{kn}(C^+)$ is 
 a periodic  quasi-solenoid  $\Lambda$. This shows that $C$ is attracted by $\Lambda$ 
 in the sense that for large $n$, $f^{kn}(C)$ is contained in a $\delta$-neighborhood of $\Lambda$. 
 To get the more precise  statement that $f^{kn}(C)\subset W^s_\delta(\Lambda)$, we have to show that 
 $W^s_\delta(\Lambda)$ is relatively open in $C^+\cap J$. 
 The argument is the same as for the local product structure: since   
 large leafwise components of $J$ are separated by some uniform distance and $C$ is thick,  
 if $x\in C\cap J$ is sufficiently close to $y\in \Lambda$, $W^s_\loc(x)\cap W^u_\loc(y)$ must belong to 
 a large component of $W^u_\loc(y)\cap J$, therefore it belongs to 
 $J^u(y)$, and we are done. 
 \end{proof}

\begin{rmk}\label{rmk:locally_connected}
Leafwise bounded components of $J$ are locally connected, as follows from Theorem~\ref{thm:topology}. On the other hand a quasi-solenoid is not locally connected, since it   locally has the structure of a Cantor set times a (locally) connected set. 
\end{rmk}

The following result says that there is a 1-1 correspondence between components of $K$ and $J$, so that the 
previous theorems yield a description of components of $K$ as well. 

\begin{prop}\label{prop:component_KJ}  
Every component of $K$ contains a unique component of $J$. 
\end{prop}
  
For polynomials in one variable, the analogous statement  is   the fact that  
every component of $K$ has a connected boundary, which follows from polynomial convexity. Here, components of $K$ have empty interior so this has to be formulated differently. 

\begin{proof}
Every component of $K$ contains a point of $J$, for otherwise it would be contained in $\mathrm{Int}(K^+)$,
so it is of the form $K(x)$ for some 
$x\in J$. If $J(x) = \set{x}$ the result is obvious. Now assume that  $J^u(x)$ is leafwise bounded. 
By Lemma~\ref{lem:leafwise_interior},  
$K^u(x)$ is obtained by filling the holes of $J^u(x)$ in $W^u(x)\simeq \C$, so $J^u(x)$ 
is equal to the intrinsic boundary of  
 $K^u(x)$ and the result follows.
  
The most interesting case is when $J(x)$ is a quasi-solenoid. Replacing $f$ by $f^k$ for some $k\geq 1$, we may assume that $J(x)$ is fixed. We proved 
 in Theorem~\ref{thm:invariant_component} that $J(x)= \bigcap_{n\geq 0} f^n(J^+_\bb(x))$. 
The very same proof shows  that $K(x) =  \bigcap_{n\geq 0} f^n(K^+_\bb(x))$. By 
Lemma~\ref{lem:component_K+J+},  
$\fr K^+_\bb(x)$ contains a unique component of $J^+_\bb(x)$ (namely, its boundary), and we conclude 
by arguing that  if   $K(x)$ contained two distinct components $J(x)$ and $J(y)$ of $J$, then 
$K^+_\bb(x)$ would contain  $J^+_\bb(x)$ and $J^+_\bb(y)$, which must be distinct 
because $\bigcap_{n\geq 0} f^n(J^+_\bb(x))\neq \bigcap_{n\geq 0} f^n(J^+_\bb(y))$, and this is impossible.  
\end{proof}

\section{Complements}

We keep the setting as in Sections~\ref{sec:semi_local} and~\ref{sec:components}. Here we prove a number of 
complementary facts which do not enter into the proof of the main theorem, so we sometimes 
 allow the presentation to be  a little sketchy.

\subsection{Transitivity}
A desirable property of quasi-solenoids is transitivity, or chain transitivity. At this stage we are not able to show that quasi-solenoidal components are transitive, but let us already explain a partial result in this direction. 
The full statement will be obtained in    Theorem~\ref{thm:complement_mixing}   under an additional assumption.  

\begin{prop}\label{prop:mixing}
If $\Lambda$ is a quasi-solenoidal component of $J$ of period $k$, there exists a quasi-solenoid 
$\Lambda'\subset \Lambda$  of period $k\ell$, which is saturated by unstable components 
(that is, if $x\in \Lambda'$ then $J^u(x)\subset \Lambda'$), 
with the property that  $f^{k\ell}\rest{\Lambda'}$ is topologically mixing. 
In addition, stable slices of $\Lambda'$ are Cantor sets   
 and  for every 
periodic point $p\in \Lambda'$, $\Lambda' = \overline{J^u(p)}$. 
\end{prop}

This proposition follows from general facts from hyperbolic dynamics. 
Let us recall some basics. 
Recall that a If $\Lambda$ is a compact hyperbolic set with local product structure, then by 
Smale's Spectral Decomposition Theorem (see e.g. \cite[\S 4.2]{yoccoz}), the non-empty 
closed invariant subset 
  $$\Omega:=   \mathcal{C}(f\rest{\Lambda}) = \overline{\mathrm{Per}(f\rest{\Lambda})}$$ 
 (where by definition $\mathcal{C}(f\rest{\Lambda})$ is the chain recurrent set of $f\rest{\Lambda}$)
 admits a decomposition    of the form  
 $\Omega = \Omega_1\cup \cdots \cup \Omega_N$. The $\Omega_i$ are called the basic pieces. 
 They are closed (and hence relatively open in $\Omega$), 
  $f$ induces a permutation on the 
  basic pieces  and if $q$ is the least integer such that $f^q(\Omega_i) = \Omega_i$, then 
  $f^q\rest{\Lambda_i}$ is topologically mixing.      In addition, $\Omega$ and the $\Omega_i$ have 
   local product structure.

\begin{proof} For notational simplicity replace $f^k$ by $f$ so that $k=1$. 
Consider the $\omega$-limit set $\omega(\Lambda) = \bigcup_{x\in \Lambda}\omega(x)$. Since a limit point is 
non-wandering, it is chain recurrent, so $\omega(\Lambda)\subset \Omega$. Conversely, since any periodic point is an $\omega$-limit point, 
we see that $\mathrm{Per}(f\rest{\Lambda}) \subset \omega(\Lambda)$, hence 
$\Omega\subset \omega(\Lambda)$ and 
$\omega(\Lambda)  = \Omega$. Then the  Shadowing Lemma  implies that 
$\Lambda \subset W^s(\Omega) = \bigcup_{x\in \Omega} W^s(x)$. 
Fix a small $\delta>0$: then $W^s(\Omega) = \bigcup_{n\geq 0} f^{-n}\lrpar{W^s_\delta(\Omega)}$. By 
   Baire's theorem,  there exists $n$ such that $f^{-n}\lrpar{W^s_\delta(\Omega)}$ 
   has  non-empty relative interior in $\Lambda$, 
 hence so does $W^s_\delta(\Omega)$, and we conclude that for some $i_0$, $W^s_\delta(\Omega_{i_0})$
 has relative non-empty interior in $\Lambda$. Let us show that $\Lambda' = \Omega_{i_0}$ satisfies the 
 requirements of the proposition. 
 
If $\ell$ is the least integer such that $f^{\ell}(\Lambda') = \Lambda'$, 
the fact that $f^{\ell}\rest{\Lambda'}$ is topologically mixing 
follows from the Spectral Decomposition Theorem. 
Since $\Lambda'$ has local product structure   and  $W^s_\delta(\Lambda')$
 has relative non-empty interior in $\Lambda$, we see that  there exists a relatively open 
 subset $U$ in $\Lambda'$ such that for any $x_0\in U$, 
  a neighborhood of $x_0$ in $J^u(x_0)$ in contained in $\Lambda'$. Since 
  $f^{\ell}\rest{\Lambda'}$ is topologically transitive we may assume that $x_0$ 
  has a dense orbit under $f^{\ell}$. So if $y\in \Lambda'$ is arbitrary we can find a sequence $(n_j)$ such that 
   $f^{\ell n_j}(x_0)\to y$. By expansion in the unstable direction, 
   there exists a uniform $\delta>0$ such that for every $j$, 
   $f^{\ell n_j}(\Lambda') = \Lambda'$ contains a $\delta$-neighborhood of $ f^{\ell n_j}(x_0)$ in 
   $J^u(f^{\ell n_j}(x_0))$, 
 so by local product structure 
   we conclude that a neighborhood of $y$ in $J^u(y)$ is contained in $\Lambda'$. On the other hand since 
   $\Lambda'$ is closed it is also relatively closed in unstable manifolds. This shows that $\Lambda'$ is saturated by unstable components. 
   
 Let us show that for every periodic point 
   $p\in \Lambda'$, $\overline{J^u(p)} = \Lambda'$. Let $N = \ell m$ be the period of $p$. 
  Since $f^{\ell}\rest{\Lambda'}$ is topologically mixing, $f^{\ell m}\rest{\Lambda'}$ is topologically
  transitive, so there exists $y$ 
   arbitrary close to $p$ such that $(f^{\ell m n }(y))_{n\geq 0}$ is dense in $\Lambda'$. 
   Let $y'$ be the projection  of $y$ in $W^u_{\loc}(p)$ under stable holonomy. By local product structure, 
   $y'$ belongs to $J^u(p)$, and  $y'\in W^s(y)$ so $(f^{ \ell m n}(y'))$ is dense, too. Since 
   all these points    belong to $J^u(p)$, we conclude that $J^u(p)$ is dense in $\Lambda'$, as asserted. 
   
For $p$ as above, since $J^u(p)$ is leafwise unbounded, it must accumulate non-trivially in $\Lambda'$. More precisely, there exists $x\in \Lambda'$ and a sequence of points $x_n\in J^u(p)$, with 
$x_n\notin W^u_\loc(x)$ and $x_n\to x$. Note that by local product structure,  $W^u_\loc(x_n)\cap \Lambda'$ 
corresponds to $W^u_\loc(x)\cap \Lambda'$ under local stable holonomy. 
Now as before there exists $y'\in W^u_\loc(p)\cap \Lambda'$  whose orbit is dense in $\Lambda'$. Thus   any $z\in \Lambda'$ is the limit of $f^{n_j}(y')$ for some subsequence 
$n_j$. But $f^{n_j}(y')$ is an accumulation point of 
$W^s_\loc(f^{n_j}(y'))\cap \Lambda'$, so the same holds for $z$, and we conclude that $\Lambda'$ is transversally perfect in the stable direction, hence it is transversally a Cantor set. 
\end{proof}

\subsection{Basins and solenoids}\label{subs:basins}  
Assume   that $f$ has an attracting cycle $\set{a_1, \ldots a_{q}}$ of exact period $q$. 
We denote by $\mathcal B$ its  basin of attraction, which is made of $k$ connected components 
$\mathcal B_i$ biholomorphic to $\cd$.  For every $i$ we can write $\mathcal B_i\cap \bb$ as the (at most) 
countable union $(\mathcal B_{i,j})_{j\geq 0}$ of its components, with $a_i\in \mathcal B_{i, 0}$.  We refer 
to these open sets as basin components and to $\mathcal B_{i, 0}$ as 
  the \emph{immediate basin}   of $a_i$. 
Note that if we replace $f$ by $f^q$, the basin of attraction of $a_i$ is now made of a single component, but  $\mathcal B_{i, 0}$ is unchanged.

By definition  a \emph{Jordan star} in $U\subset \C$ 
is a finite union of simple Jordan arcs in  $U$, intersecting at a single  point. 

\begin{thm}\label{thm:basin}
Let $f$ be dissipative and hyperbolic with a disconnected and stably totally disconnected Julia set. 
Suppose that $f$ admits an attracting  fixed point with immediate basin $\mathcal B_0$. Then:
\begin{enumerate}[(i)]
\item $\fr\mathcal B_0$ is a   properly immersed topological 
submanifold of dimension 3, which intersects any global unstable 
transversal in finitely many Jordan domains.
\item $\bigcap_{n\geq 0} \fr\mathcal B_0$ is a quasi-solenoid, whose unstable slices are Jordan stars.  In particular there is a (saddle) periodic point in $\fr\mathcal B_0$. 
\end{enumerate}
\end{thm}

We can be more precise about the structure of $\fr\mathcal B_0$: locally it is homeomorphic to the product of a 
2-disk by a  Jordan star. The proof of the theorem shows that if the components of 
$\mathcal B_0\cap \Delta^u$ have disjoint closures, then these stars are reduced to Jordan arcs, 
that is, $\fr\mathcal B_0$ is a   topological  submanifold.

The following basic fact is crucial for the proof. 

\begin{lem}\label{lem:respect_basin}
The stable lamination $\cW^s$ respects basin boundaries. That is, if $x\in J^+$ belongs to the boundary of an attracting basin $\mathcal B$, then so does its image under stable holonomy. 
\end{lem}

\begin{proof}
This follows readily from the existence of a local extension of the stable lamination (Lemma~\ref{lem:extension}):  indeed if a leaf 
of the extended foliation joined a point from $\mathrm{Int}(K^+)$ to a point of $(K^+)^\complement$, it would 
have to intersect $J^+$. (See also \cite{saddle}, Step 3 of the proof of the main theorem, for an alternate argument without extending the stable lamination.)
\end{proof}

\begin{proof}[Proof of Theorem~\ref{thm:basin}]  
Fix a global unstable transversal $\Delta^u$. Since every semi-local stable manifold intersects 
$\Delta^u$, $\mathcal B_0\cap \Delta^u$ is   non-empty, and by the Maximum Principle each of its connected 
components is a topological disk. Pick such a connected component $\Omega_0$.  
By the John-Hölder property (Theorem~\ref{thm:john-holder_basin}),   $\fr\Omega_0$ is locally connected,
 and   by the Maximum Principle again there is no cut point, and it follows that 
 $\Omega_0$ is a Jordan domain (see~\cite[Thm 2.6]{pommerenke}). 

If the  diameter of $\Omega_0$ is small then, 
by Remark~\ref{rmk:general_thick_thin}, enlarging $\bb$ if necessary 
the saturation $\widehat{\fr\Omega_0}$ of $\fr\Omega_0$ by semi-local stable leaves is 
topologically a product and we infer that   $\widehat{\fr\Omega_0}\cap \Delta^u$ has finitely many components. Otherwise the diameter is large and by the same remark, every component of  
$\widehat{\fr\Omega_0}\cap \Delta^u$ has a large diameter. 
Then the finiteness of the number of such components follows from the John-Hölder property of $W^u(x)\setminus K^+$, Proposition~\ref{prop:finiteness}, 
and the finiteness statement for  interior components in Lemma~\ref{lem:diameter}.

By the Maximum Principle, if $\Omega_0$ and $\Omega_1$ are two components of $\mathcal B_0\cap \Delta^u$ such that $\overline{\Omega_0}\cap\overline{\Omega_1}\neq \emptyset$, then 
$\overline{\Omega_0}\cap\overline{\Omega_1}$ is a single point. Indeed if this set contained two distinct points 
$z$ and $z'$, by  using 
crosscuts of $\Omega_0$ and $\Omega_1$ ending at $z$ and $z'$
we could construct a Jordan domain $U$ with  $\fr U \subset \overline{\Omega_0}\cup\overline{\Omega_1}$, 
and  $U$ would be contained in the Fatou set, a contradiction. Create a plane 
graph from $\mathcal B_0\cap \Delta^u$ whose vertices are  its components and edges are added when two 
components touch. The Maximum Principle again shows that this graph is a finite union of trees. Since the 
stable holonomy respects  $\fr\mathcal B_0$ and  $\fr\mathcal B_0$ is obtained from $\fr\mathcal B_0\cap \Delta^u$ by saturating by stable manifolds, 
the description of $\fr\mathcal B_0$ as a properly immersed topological 
submanifold of dimension 3 follows. 

The proof of the second item of the theorem is similar to that of Theorem~\ref{thm:invariant_component}. 
First, $\fr\mathcal B_0$ is connected: the argument is identical to that of Lemma~\ref{lem:component_K+J+}. 
Then,  for every $x\in \fr\mathcal B_0 \cap J^-$,   there are only finitely many components 
of $\mathcal B_0\cap D^u(x,1)$ (resp. $\fr\mathcal B_0\cap D^u(x,1)$) intersecting $D^u(x, 1/2)$. 
Indeed,    observe first  that it is enough to prove this in  
$D^u(x,r)$ for some uniform $r$. By the
uniform boundedness of the degree of semi-local stable manifolds in $\bb$, 
 there is a uniform $r$ such that 
$D^u(x,r)$ can be pushed to $\Delta^u$ by stable holonomy, and the applying 
   item~(i) of the theorem completes the argument. 
From this point we proceed exactly as in   Theorem~\ref{thm:invariant_component}. 
The existence of a periodic point in $\fr\mathcal B_0$ follows 
from general hyperbolic dynamics (see the comments after Proposition~\ref{prop:mixing}). 
\end{proof}

\begin{rmk}\label{rmk:boundary_component}
It follows from this description that if $x\in \Lambda$ lies at the boundary of $\mathcal B_0$, then 
in $W^u(x)$, $x$ belongs to the boundary of a component $\Omega$
of $\mathcal B_0\cap W^u(x)$. In particular, $\Omega$ is a Fatou disk contained in $\comp_K(x)$. 
\end{rmk}

\begin{rmk}
We do not know whether components of $\mathcal B_0\cap \Delta^u$ can actually bump into each other, or equivalently if $\bigcap_{n\geq 0} \fr\mathcal B_0$ does contain stars. 
If bumping   occurs, let $E$ be the finite set of points at which the closures of the
components of $\mathcal B_0\cap \Delta^u$ touch each other. Then $W^s_\bb(E)$ is a finite union of 
vertical submanifolds, and $f(W^s_\bb(E))\subset W^s_\bb(E)$. It follows that 
$\bigcap_{n\geq 0} f^n(W^s_\bb(E))$ is a finite set of periodic points, and for any other point $x$ in the limiting 
quasi solenoid $\Lambda:= \bigcap_{n\geq 0} \fr\mathcal B_0$,    
  $\Lambda\cap W^u_\loc(x)$ is a Jordan arc. Thus, roughly speaking, 
  \emph{$\Lambda$ has the structure of finitely many solenoids attached at periodic ``junction'' points.} 
\end{rmk}

\subsection{Branched Julia set model}\label{subs:branched}

Let $\Lambda$ be a quasi-solenoidal component of $J$, and 
without loss of generality   assume that $\Lambda$ 
is fixed. Let $J^+_{\bb}(\Lambda)$ be its connected component in $J^+_\bb$ and consider its intersection 
$D:=J^+_{\bb}(\Lambda)\cap \Delta^u$  with  some unstable transversal, which is made of 
finitely many thick components. Introduce a  relation $\sim$ on $D$ by $x\sim y $ if and only if 
$W^s_{\overline{\bb}}(x)  = W^s_{\overline{\bb}}(y)$, where by definition 
$W^s_{\overline{\bb}}(x)  =  \bigcap_{\e>0} W^s_{(1+\e)\bb}(x)$. 
 Equivalently $x\sim y$ iff 
 $\overline{W^s_\bb(x)}\cap  \overline{W^s_\bb(y)} \neq \emptyset$: concretely, this means that $x$ and $y$ are related when they are connected by a stable manifold which is tangent to $\fr\bb$. 
 This defines a closed equivalence relation on $D$. We denote by $\tilde D:= D/\sim$ the 
 quotient topological space, which is compact (and Hausdorff) and by $\pi:D\to \tilde D$ the natural projection. 
 Since $f(W^s_{\overline{\bb}}(x))\subset W^s_{\overline{\bb}}(f(x))$, $f$ descends to the quotient 
 $\tilde D:= D/\sim$ to a well defined continuous map $\tilde f$. 
 
 Geometrically $\tilde D$ has to be thought of as a \emph{branched Julia set}, lying on the branched  
  surface --in the sense of Williams~\cite{williams}--
obtained by collapsing the semi-local stable leaves of the extended stable lamination.
 Then $\tilde f$ is expanding on the plaques of this branched manifold\footnote{Here by plaque we mean one of the finitely many overlapping disks which make up  a local chart of a branched manifold, see \cite[Def. 1.0]{williams}}, and 
its iterates  are uniformly quasiconformal wherever defined, since they are obtained by iterating $f$ and projecting along the stable lamination. Observe that $f$ is not necessarily surjective, since for every $x\in D$,
 $f^n(x)$ eventually belongs to $W^s_\bb(\Lambda)$, which may be smaller than $J^+_{\bb}(\Lambda)$ (cf. Figure~\ref{fig:bidisk_stable_component}). On the other hand 
 by the last assertion of Theorem~\ref{thm:component_J}, there exists a uniform $N$ 
 such that $f^N(J^+_{\bb}(\Lambda))\subset W^s_\bb(\Lambda)$. 
 It follows that the sequence $\bigcap_{0\leq k\leq n} \tilde f^k(\tilde D)$ is stationary for $n\geq N$ and 
that   $\tilde D': = \pi(W^s_\bb(\Lambda)\cap \Delta^u)$, 
  is an invariant,  closed, and plaque-open subset of 
 $\tilde D$ on which $\tilde f$ is surjective. 

\begin{prop}
With the above definitions, the  dynamical system $(\Lambda, f)$ is topologically conjugate to the natural extension of $( \tilde D, \tilde f)$ 
(or equivalently $(\tilde D', \tilde f)$). 
\end{prop}

\begin{proof}
Indeed define $h: \varprojlim (\tilde D, \tilde f) \to \Lambda$ by $h((\tilde x_n)_{n\in \Z})  = \bigcap_{n\geq 0} f^n(W^s_{\overline\bb}(x_{-n}))$, whose  inverse is $y\mapsto h\inv(y) = ((W^s_{\overline\bb}(f^n(y)))_{n\in \Z}$. 
\end{proof}

\section{Non-divergence of holonomy and applications}\label{sec:NDH}
 
 \subsection{The NDH property}\label{subs:NDH}
  We say that the property of  \emph{Non-Divergence of Holonomy} (NDH) holds if for every pair of points 
  $x,y \in J$ such that 
 $y$ belongs to $W^s(x)$, 
 the stable  holonomy, which is locally defined from a neighborhood of 
 $x$ in $W^u(x)$ to a neighborhood of $y$ in $W^u(y)$, can be continued along any path contained 
 in $J^u(x)$. 

\begin{rmk}\
\begin{enumerate}
\item The stable holonomy $h:W^u(x)\to W^u(y)$ is independent of the choice of a path $c$ from $x$ to $y$ in $W^s(x)$ because 
$W^s(x)$ is simply connected. 
\item An unstable component $J^u(x)$ is typically \emph{not} simply connected 
(since it may encloses the trace of an attracting basin on 
$W^u(x)$). So even if the stable holonomy from $x$ to $y$ admits an extension along continuous paths, it does not generally yield  
a well-defined map from $J^u(x)$ to $J^u(y)$. 
\end{enumerate}
\end{rmk}

 We do not know any example where the NDH property fails. An analogue of this 
  property was studied in the context of the classification of Anosov diffeomorphisms, where it is expected to 
 be a crucial step in the classification program. 
It was established in the two dimensional case in  \cite{franks} 
(see also~\cite{brin, kleptsyn-kudryashov} for related results). 

Back to automorphisms of $\cd$, we have the following simple criterion:

\begin{prop}\label{prop:NT_NDH}
A  sufficient condition for  the NDH property is that   
the stable lamination $\mathcal W^s$ of $J^+$ is transverse to $\fr\bb$ (No Tangency condition, NT).   
\end{prop}

\begin{proof}
Assume that the No Tangency condition holds and let 
 $x,y\in J$ be such that $y$ belongs to $W^u(x)$. Replacing $x$ and $y$ by $f^k(x)$ and $f^k(y)$ for some 
positive $k$, we may assume that $y\in W^s_\bb(x)$. There is a germ of stable holonomy $h$ sending a neighborhood of $x$ in $J^u(x)$  to some neighborhood of $y\in J^u(y)$. Let $\gamma:[0,1]\to J^u(x)$ be a continuous path: we have to show that $h$ can be continued along $\gamma$. For this, introduce 
 $E\subset [0, 1]$   the set  of parameters  $t$ such that 
$h$ can be continued along  $\gamma\rest{[0,t]}$ and $h(\gamma(t))\in W^s_\bb(\gamma(t))$.  
Obviously, $E$ is a relatively open subinterval of $[0,1]$ containing 0, 
and the proof will be complete if we show that $E$ is closed.   Thus, assume that 
$(t_n)\in E^\N$ is an increasing  sequence converging to $t_\infty$, and 
let $y_\infty$ be any cluster value of the sequence $(h(\gamma(t_n)))$.
The main observation  is that since $\mathcal W^s$ is transverse to $\fr\bb$, $W^s_\bb (\gamma(t_n))$ converges to 
$W^s_\bb (\gamma(t_\infty))$ in the Hausdorff topology,   with multiplicity 1, or equivalently in the $C^1$ topology. Furthermore, 
by the uniform boundedness of the vertical degree, there is a uniform $L$ such that for every $n$, there is a path of length at most 
$L$ joining $\gamma(t_n)$ to $h(\gamma(t_n))$ in $W^s(\gamma(t_n))$. It follows that the assignment 
$\gamma(t_n)\mapsto h(\gamma(t_n)$ is equicontinuous. Let $y_\infty$ be a cluster value of $( h(\gamma(t_n)))$. The equicontinuity property shows that $h(\gamma(t_n))$ actually converges to $y_\infty$, and also that the  the 
points $h(\gamma(t_n))$ belong to the same local plaque of the unstable lamination, which must thus coincide with $W^u_\loc(y_\infty)$. From this we conclude that $h$ extends to a neighborhood of $\gamma(t_\infty)$, with 
$h(\gamma(t_\infty)) = y_\infty$, and we are done. 
\end{proof}

One may argue that the NT condition is not intrinsic since it depends on the choice of the bidisk $\bb$. 
To get around this issue  we 
may consider the following variant: 

   (NT$_G$)   there exists $R>0$ such that the stable foliation admits no tangency with the hypersurface $\set{G^- = R}$.   
 
Note that the level set $\set{G^- = R}$ is smooth near $J^+$ for every $R>0$: indeed by the local structure of $G^-$ near infinity
this is the case  when $R$ is large, and then we use invariance to propagate this property to all $R>0$.   Arguing exactly as in the 
previous proposition shows that the  NT$_G$ property implies NDH. 

Using this idea   also enables us to understand more precisely how the NDH property may fail. If $x$ and $y$ are two points in $J$ with  $y\in W^s(x)$, define the 
\emph{Green distance} 
$$d_{G}(x,y):= \inf_{c: x\to y} \max(G^-\rest{c})$$
where the infimum runs over the set of continuous paths $c:[0, 1]\to W^s(x)$ joining $x$ to $y$. 
Since  $W^s(x)\cap J$ is totally disconnected, this indeed defines an ultrametric on $W^s(x)\cap J$, which is uniformly contracted by $f$: $d_{G}(f(x),f(y)) = d\inv d_G(x,y)$. 
It provides   an intrinsic way of measuring 
how far we need to go in $\cd$ to connect  two unstable components by  stable manifolds. Arguing exactly as in Proposition~\ref{prop:NT_NDH} shows:

\begin{prop}
Let $x,y\in J$ with $y\in W^s(x)$ and denote by $h$ the germ of stable
 holonomy $h:W^u_\loc(x)\to W^u_\loc(y)$. Let $\gamma:[0,1]\to J^u(x)$ be a continuous path and assume that $h$ can be continued along $\gamma([0, t^\star))$. Then $h$ admits an extension to $t^\star$ if and only if 
 $d_G(\gamma(t), h(\gamma(t)))$ is bounded as $t\to t^\star$.  
\end{prop}

\subsection{No queer components}

\begin{thm} \label{thm:no_queer_components}
Let $f$ be   dissipative and hyperbolic, with a disconnected and 
stably totally disconnected Julia set. Assume further that the NDH property holds. 
Then any non-trivial periodic component of $K$ contains an attracting point. 
\end{thm}

\begin{proof} We argue by contradiction: assume that  $\Lambda$ is a component of $K$ which does not 
contain any attracting periodic point. Let $C$ be the component of 
$\Lambda$ in $K^+\cap \bb$. Our hypothesis 
implies that $C$ has empty interior, so $C$ is a component of $J^+\cap \bb$ (and $\Lambda$ is a component of 
$J$.  Fix an unstable transversal $\Delta^u$ and let $E$ be a component of $C\cap \Delta^u$, which must have empty interior in $ \Delta^u$ by Lemma~\ref{lem:fatou_disk}. Thus $E$ is a
locally connected continuum with empty interior, that is, a dendrite. 

\begin{lem}\label{lem:fixed_point}
For every $x\in E$, $W^s(x)\cap E = \set{x}$. 
\end{lem}

Assuming this lemma for the moment, let us complete the proof. By the expansion in the unstable direction, for every $x\in E$, there exists $\delta_1>0$ such that for every $n\geq 0$, $f^n(E)$ is not relatively compact in $D^u(f^n(x), \delta_1)$, and by the John-Hölder property, there exists $\delta_2>0$ such that any two components 
of  $f^n(E)$ in $D^u(f^n(x), \delta_1)$ intersecting $D^u(f^n(x), \delta_1/2)$ are $\delta_2$-separated. 
Fix a covering of $J$ by unstable flow boxes. By the product structure of $J$, there 
exists $\e>0$ such that if  $y,z\in f^n(E)$ are $\e$-close in $\C^2$ but not on the same unstable plaque, 
then the components $\comp_{f^n(E)\cap D^u(y, \delta_1)}(y)$ and $\comp_{f^n(E)\cap D^u(z, \delta_1)}(z)$
are related by local stable holonomy. Finally, by expansion along the unstable direction and the previous 
separation property,  $f^n(E)$ cannot be 
contained in boundedly many unstable plaques as $n\to\infty$. Thus, for sufficiently large $n$
 we can find  two points in $f^n(E)$  which are $\e$-close in $\C^2$ but not on the same unstable plaque, so there exists  $y\in f^n(E)$ such that $W^s_\loc(y)$ intersects $f^n(E)$ in another point. This contradicts 
 Lemma~\ref{lem:fixed_point} and we are done. 
\end{proof}

\begin{proof}[Proof of Lemma~\ref{lem:fixed_point}]
Assume that $W^s(x)\cap E $ contains another point $y\neq x$. Then the stable holonomy defines a germ of 
homeomorphism $h: E\cap U_x\to E\cap U_y$, where $U_x$ is some neighborhood of $x$ (resp. $y$). By the 
NDH property, $h$ can be continued along paths in $E$. Since $E$ is simply connected, this extends to a  
globally defined map $h:E\to E$, sending $x$ to $y$, which is a local homeomorphism, hence a covering, 
so again using the fact that $E$ is simply connected, we conclude that $h$ is a homeomorphism.

It is a classical fact that any continuous self-map of $E$ admits a fixed point. For the reader's convenience let us include the argument. View $E$ as a subset of the plane. Then, by the Carathéodory theorem, the Riemann map $\C\setminus \dd \to \C\setminus E$ extends to a continuous 
and surjective map $\fr\dd\to\fr E = E$. From this we can construct a  
 topological disk $U \supset E$ and a retraction 
$r: \overline U \to E$: indeed take the disk bounded by some equipotential and define $r$ as collapsing each 
 external ray to its endpoint. Now let  $g = h\circ r$. Since $g$ maps $\overline U$ into itself, by the Brouwer 
 fixed point theorem it admits a fixed point $x_0$. Finally, since $g(\overline U) \subset  E$, $x_0$ belongs to
  $E$, so $g(x_0) = h(r(x_0))= h(x_0) = x_0$. 

To conclude the proof we show that the existence of such a fixed point contradicts the hyperbolicity of $f$. For this, fix a continuous path  $(x_t)_{t\in [0, 1]}$ joining $x_0$ to $x_1:=x$ and let 
$t^\star = \max \set{t\in [0,1],  \ h(x_t)=x_t}$, which satisfies $0\leq t^\star <1$. 
As $t>t^\star$ tends to $t^\star$, we see that the two point set  $\set{x_t, h(x_t)}$  collapses to 
$\set{x_{t^\star}}$. This means that there is a tangency between the stable lamination and 
$\Delta^u$ at $x_{t^\star}$, which is the desired contradiction. 
\end{proof}

\begin{rmk}\label{rmk:wrap}
With notation as in the proof of the theorem, 
it is not difficult to deduce from the proof that for every $\delta>0$, for $n\geq n(\delta)$
there exists a non-trivial simple closed curve contained in $W^s_\delta(f^n(E))$. So by the last assertion of Theorem~\ref{thm:component_J}, there is a non-trivial  simple closed curve contained in 
$W^s_\delta(\Lambda)$. 
Without the NDH property, we cannot exclude a    situation  where these simple closed curves 
do not enclose an attracting basin. 
 We may qualify these dendrites and their limit sets as \emph{queer components} of $J$. So Theorem~\ref{thm:no_queer_components} 
 asserts that under the NDH property, \emph{queer components of $J$ do not exist}. 
\end{rmk}

\subsection{Topological mixing}

\begin{thm}\label{thm:complement_mixing}
 If the NDH property holds, if $\Lambda$ is a  quasi-solenoidal component of period $k$, then 
 $f^k\rest{\Lambda}$ is topologically mixing.  In particular $\Lambda$ is transversally a Cantor set.
\end{thm}
 
\begin{proof} Without loss of generality we may assume $k=1$. 
We resume Proposition~\ref{prop:mixing} and its proof.  
Let $\Lambda'$ be as in Proposition~\ref{prop:mixing}, and let us show that $\Lambda' = \Lambda$. 
Since $\Lambda'$ is saturated in the unstable 
direction, $W^s(\Lambda')$ is relatively open in $\Lambda$. The NDH property shows that if 
$y\in W^s(\Lambda')$, then $J^u(y)\subset W^s(\Lambda')$: indeed the set of points   $z\in J^u(y)$ such that 
$z\in W^s(\Lambda')$ is open because $W^s(\Lambda')$ is relatively open, and since $J^u(y)$ is arcwise 
connected, the the NDH property implies that it is closed as well. Thus by the local product structure of $\Lambda$, we conclude that $W^s(\Lambda')$ is relatively closed  in $\Lambda$, and by connectedness 
we conclude that  $W^s(\Lambda') =\Lambda$. 

Fix a small $\delta>0$. By Baire's theorem, 
 we infer that $f^{-n}(W^s_\delta(\Lambda'))$  has non-empty relative interior in $\Lambda$ for large $n$, 
 hence so does $W^s_\delta(\Lambda')$ by invariance. Arguing as in Proposition~\ref{prop:mixing}, we see that  by topological transitivity, $W^s_\delta(\Lambda')$
is actually relatively open  in $\Lambda$. Therefore   $\bigcup_{n\geq 0} f^{-n}\lrpar{W^s_\delta(\Lambda')}$ is an open cover of $\Lambda$ and  by compactness  we conclude   that  
  $\Lambda$ is contained in $\bigcup_{0\leq n\leq n_0} f^{-n}\lrpar{W^s_\delta(\Lambda')}$ for some $n_0$. 
   and   since $f^{n_0}(\Lambda) = \Lambda$ we finally deduce 
    that  $ \Lambda \subset W^s_\delta(\Lambda')$. Since $\delta$ was arbitrary,
  $\Lambda\subset\Lambda'$, and we are done. 
  \end{proof}

\begin{rmk}
A similar argument shows that under the NDH property, the quasi-solenoids obtained as limit sets of basin boundaries in Theorem~\ref{thm:basin} are transitive. 
\end{rmk}

As a consequence of transitivity we can be more precise about the 
topological structure of periodic components of $K$. 

\begin{prop} \label{pro:complement_no_queer_components}
Let $f$ be   dissipative and hyperbolic, with a disconnected and 
stably totally disconnected Julia set. Assume further that the NDH property holds. 
 Then for any non-trivial component $D$ of $K$,  $D\cap \mathrm{Int}(K^+)$ is dense in $D$.  Equivalently, for any $x\in D$, $D\cap W^u(x)$ is the closure of its interior for the intrinsic topology. 
\end{prop}
 
\begin{proof}
The equivalence between the two assertions follows from Lemma~\ref{lem:fatou_disk}, Lemma~\ref{lem:respect_basin}, and the local product structure. Let  $D$ be as in the statement of the 
proposition and  $C$ be its component in $K^+\cap\bb$. Let also $\Lambda$ the unique component of $J$ contained in $D$ (Proposition~\ref{prop:component_KJ}). 
Without loss of generality we may assume that 
$D$ (hence $C$ and $\Lambda$) is fixed by $f$. By Theorem~\ref{thm:no_queer_components} 
$D$ contains an attracting periodic point $a$, so the immediate basin $\mathcal B_0$ 
of $a$ is contained in $C$. By Theorem~\ref{thm:basin}, $\fr B_0$ contains a saddle periodic point $p$, 
which must belong to $\Lambda$ (indeed by Lemma~\ref{lem:component_K+J+} and 
Theorem~\ref{thm:invariant_component}, 
$\Lambda = \bigcap_{n\geq 0}f^n(\fr C)$). The topological mixing of 
$f\rest{\Lambda}$ (Theorem~\ref{thm:complement_mixing}) classically 
implies that   $W^s(p)\cap \Lambda$ is dense in $\Lambda$.  
 Indeed let $U$ be a product 
  neighborhood of $p$ in $\Lambda$, and $V$ be an arbitrary open subset of $\Lambda$. 
 Then   for sufficiently large $q\geq 0$ there exists  $y_q\in V$ such that $f^q(y_q)\in U$. Since $\Lambda$ has local 
 product structure $[f^q(y_q), p]:=W^u_\loc(f^q(y))\cap W^s_\loc(p)$ belongs to $\Lambda$, hence increasing $n$ again if needed, $z_q:=f^{-q}([f^q(y_q), p])$ is a point in  $W^s(p)\cap V$. 
   
To conclude from this point, we observe that by 
Remark~\ref{rmk:boundary_component} (applied to  $f^{-q}(\mathcal B_0)$)
$z_q$ belongs to the boundary of a component $\Omega$ of $W^u(z_q)\cap f^{-q}(\mathcal B_0)$ 
contained in $D$, and  we are done. 
\end{proof}

\subsection{Concluding remarks}

The non-existence problem for queer components 
bears some similarity with another well-known open problem: the non-existence of 
Herman rings for complex Hénon maps (see~\cite{bs2} for an early account). 
Indeed assume that  
$f$ admits a Herman ring, that is, a Fatou component $\Omega$ biholomorphic to 
the product of an annulus times $\C$. More precisely there exists a biholomorphism 
$h: \Omega\to A\times \C$, where $A$ is a standard 
 annulus, which conjugates $f$ to $(x,y)\mapsto (e^{i\theta} x, \delta y)$, $\abs{\delta} <1$. 
Assume further that   $J$ is disconnected, and 
fix an unstable transversal $\Delta^u$ (recall that its existence 
 does not require $f$ to be hyperbolic). Then if $C$ is an invariant circle in $A$, $f$ admits an invariant ``cylinder'' $\mathcal C = h\inv(C\times \C)$. Any component of 
$\mathcal C\cap \Delta^u$ is a piecewise smooth immersed curve, 
and a contradiction would follow if we can show that it bounds a disk in $\Delta^u$ (since by the maximum principle this disk would be a Fatou disk, whose   normal limits   would fill up the annulus).  
In other words, if $f$ admits a Herman ring, $\mathcal C\cap \Delta^u$ 
is a countable union of dendrites whose saturation under the stable foliation of $\mathcal C$ bounds a disk, 
but not a holomorphic disk (compare with Remark~\ref{rmk:wrap}). 
 Note however that   a limitation to 
the  analogy between the two problems  is that   
 the NDH property holds trivially in the Herman ring case, so the difficulty is of a different nature.

\appendix

\section{The core of a quasi-solenoid} \label{app:core}

In this Appendix,  we sketch the construction of the \emph{core} of a quasi-solenoidal component, 
which should intuitively 
be understood as the space obtained from this component after removing all ``bounded decorations'' in 
unstable manifolds. 
Initially designed  as a potential 
 tool  to  prove the non-existence of queer quasi-solenoids, it  
 also gives interesting information on the combinatorial structure of tame ones. 
 It would be 
 interesting to compare it with other constructions such as Ishii's 
 Hubbard trees (see \cite{ishii_survey}). We keep the setting as in the previous sections, that is 
 $f$ is a uniformly hyperbolic dissipative Hénon map, with a a disconnected and stably totally disconnected Julia set. 
 
 \subsection{Number of accesses} The discussion in this paragraph is 
 reminiscent from  \cite[\S 7]{bs7}, which deals with the connected case. 
Pick $x\in J$.  For any $R>0$, define   $N^u(x, R)$ to be 
  the number of connected components $\Omega$ of 
$D^u(x, R)\setminus J$ such that $x\in \overline \Omega$. Since $K\cap D^u(x,R)$ 
has the John-Hölder property,  Corollary~\ref{cor:fast_escaping}   implies that 
$N^u(x, R)<\infty$. Thus, $R\mapsto N^u(x, R)$ is a integer-valued 
non-increasing function which drops when two  components of $ D^u(x, R)\setminus J$ 
merge. The limit 
$$N^u_\loc(x):=\lim_{R\to 0} N^u(x, R)$$ is the number of local accesses to $x$, and 
 $$N^u(x):=\lim_{R\to \infty} N^u(x, R)$$ is the number of connected components of 
 $W^u(x)\setminus J$. Note that if $J^u(x)$ is bounded then $ N^u(x)  =1$, so this
  notion is interesting only when $x$ belongs to a quasi-solenoidal component.

 We can also restrict to counting 
  accesses from infinity, that is  components of $D^u(x, R)\setminus K^+$, and we obtain corresponding numbers $N^u_\infty(x,R)$, $N^u_{\infty, \loc}(x)$ and  $N^u_\infty(x)$. We have that 
  $N^u_\infty(x) \leq N^u(x)$ (and similarly for the other quantities), 
  and, since every point of $J$ is accessible from infinity, 
  $N^u_\infty(x)\geq 1$.  
(\footnote{The John-Hölder property of the basin of infinity directly guarantees the finiteness of  
 $N^u_{\infty, \loc}(x)$, but not that of   $N^u_\loc(x)$ (see Remark~\ref{rmk:eta0}). 
 This property can  actually be salvaged as follows: if for small $R$, $N^u(x, R)$ is large, then
 for some $k\gg 1$, $N^u(f^k(x), 1)$ is large, and projecting to some fixed transversal yields a contradiction.}) 
 
\begin{lem}\label{lem:N^u_lsc}
$N^u$ (resp. $N^u_\infty$) is upper  semicontinuous on $J$, 
that is, for any $k\geq 1$, $\set{x, \ N^u(x)\geq k}$ is closed. 
\end{lem}
  
\begin{proof}
We deal with $N^u$, the proof for    $N^u_\infty$ is similar. 
It is enough to assume that $k\geq 2$. 
By the local product structure of $J$, it is enough to 
  study the semi-continuity of $x\mapsto N^u(x)$ separately along stable and unstable  manifolds. 
  Let us start by studying this semicontinuity along a local stable transversal. 
We have to prove that $\set {x, \ N^u(x)<  k}$ is open. 
Indeed assume that there are $j < k$   accesses to $x$ in $W^u(x)\setminus J$. This means that 
for large $R$, $D^u(x, R) \setminus J$ has $j$ connected components 
accumulating at $x$. 
If  $x'\in W^s(x)$  then the local stable holonomy between 
$W_\loc^u(x)$ and $W^u_\loc(x')$ is a homeomorphism,  which locally preserves the number 
of components of $W^u_\loc(x)\setminus J$.  In addition if $x'$ is sufficiently close to $x$, this 
holonomy is defined in $D^u(x,R)$. Indeed for this 
it is enough to 
iterate backwards until $f^{-n}(D^u(x, R))$ is 
contained in the domain of the extended stable lamination. 
Therefore,   there is a large domain $D'$ in  $W^u(x')$ such that $D'\setminus J$ has 
$j$ connected components  accumulating  on $x'$. Since the number of components may drop when enlarging  this disk further, we conclude that   $N^u(x')\leq j$. 

Now we work inside a given unstable manifold. Let $R$ be such that 
$N^u(x, s) = N^u(x) = j$ for $s\geq R-1$. 
By the Hölder-John property, for $R'<R$, 
 $D^u(x, R)\setminus J$ admits finitely many components intersecting $D^u(x, R')$. 
 So if   $N^u(x) = j$, there is some $0<\e< 1$ such that only $j$ of these components reach 
 $D^u(x, \e)$, and we conclude that for $x'\in D^u(x, \e)$, $N^u(x', R-1) \leq j$, 
 hence $N^u(x')\leq j$, as asserted. 
\end{proof}

Since $f$ acts linearly on unstable parameterizations, 
 $N^u(x, R) = N^u(f(x), \lambda^u R)$, and we obtain: 
 
\begin{cor}\label{cor:omegaNu}
If $N^u_\loc(x)\geq k$ then for any $y\in \omega (x)$, $N^u(y)\geq k$. 
\end{cor}

An  argument similar to that of the second part of Lemma~\ref{lem:N^u_lsc} implies  
(compare~\cite[pp. 490-491]{bs7}):

\begin{lem}\label{lem:discrete}
For any $R>0$ and any $x\in \Lambda$,  the set $\set{y\in W^u(x), \ N^u(y,R) \geq 3}$ is discrete for the 
intrinsic topology. 
\end{lem}

\begin{prop}
The set $\set{x\in J,\  N^u(x)\geq 3}$ is a finite set of saddle periodic points. 
\end{prop}

\begin{proof} 
By Lemma~\ref{lem:discrete},  the set $\set{x\in J, \ N^u(x, R)\geq 3}$ 
is contained in a countable union of local stable manifolds. Since any point in $J$ can be joined to a given unstable transversal $\Delta^u$ by a stable path of uniform length, by 
taking small enough $R$ we infer that the projection of this set to $\Delta^u$  is actually finite. 
Therefore,  the set  $\set{x\in \Lambda, \ N^u(x)\geq 3}$  is a closed 
invariant set contained in 
a finite union of semi-local stable manifolds, so it is finite. 
\end{proof}

\subsection{Definition(s) and properties of the core} 
Let $\Lambda$ be a quasi-solenoidal component of $J$. 
There are  several possible 
 definitions for the core of $\Lambda$. It is unclear for the moment
 which choice is the most appropriate.  We define:
 \begin{itemize}
 \item $\core(\Lambda) = \set{x\in \Lambda, \  N^u(x)\geq 2}$
 \item $\core'(\Lambda)  = \omega\lrpar{\set{x\in \Lambda, \ N^u_\loc(x)\geq 2}}$
 \end{itemize}
By   Corollary~\ref{cor:omegaNu} we have the inclusion
$\core'(\Lambda) \subset \core(\Lambda)$, and it is an open problem 
whether equality holds  It is obvious from the definition that 
$\core(\Lambda)$ (resp. $\core'(\Lambda)$)
is invariant  and  Lemma~\ref{lem:N^u_lsc} implies that it is closed. Hence it is a 
closed hyperbolic set.  Another natural open question is whether $\core(\Lambda)$ is connected. 

The core of the Julia set is the union of the cores of its  finitely many quasi-solenoidal components. 
If $x\in J$ is any point such that $W^u(x)\setminus J$ has 
 several local accesses, then $\omega(x) \subset \core( J)$. 
 
We say that $x\in \core(\Lambda)$ is \emph{regular} if $N^u(x)=2$ and \emph{singular} otherwise. 
Recall that the singular set is a finite set of periodic points. Note that if $x$ belongs to the core, then 
$J^u(x)$ disconnects $W^u(x)$. 

\begin{conjecture}
$\core(\Lambda)$ has local product structure near any regular point, and is locally the product of a Jordan arc by a totally disconnected set. 
\end{conjecture}

On the other hand, 
$\core(\Lambda)$ does not have local product structure in the neighborhood of any of its singular points, unless it is locally contained in a single unstable manifold. So the structure of the core should 
be that of a union of solenoids joined at finitely many  branch points.
It seems that in the example 
described in~\cite[Thm 4.23]{ishii-nonplanar}, one quasisolenoidal component has a core made of two solenoids attached at a fixed saddle point. 

Note that if $\Lambda$ is not a queer component, that is the associated 
component of $K$ contains an attracting periodic point, then the solenoid at the boundary of the 
immediate basin, constructed in \S~\ref{subs:basins}, is contained in the core. Indeed it is obtained by 
taking   limits of Jordan arcs locally separating an attracting basin from the basin of infinity.  
So the topological structure of the 
 core   should  give an account how these various basins are organized and attached to each other 
  in $\Lambda$ (compare with the Hubbard tree in one-dimensional dynamics).

Finally, we may also define $\core_\infty(\Lambda) = \set{x\in \Lambda, \  
N^u_\infty(x)\geq 2}$. (If $\Lambda$ is a queer component, then $\core_\infty(\Lambda)  = \core(\Lambda)$.) We expect that  $\core_\infty(\Lambda)$ is a finite set. 
Indeed, if not, it should contain a Jordan arc
 such that every point is accessible from both sides by the basin of infinity, and such arcs 
  should not exist. Indeed,  iterating forward,
  and arguing as in Theorem~\ref{thm:no_queer_components}, 
a large iterate of this arc must spiral  and  come
close to itself, hence,  projecting to an unstable transversal, this  would cut out a
 Fatou disk,  and we conclude that  one side of the arc is contained in  an attracting basin.

\section{Continuity of   affine structure} 

Here we present the following  mild generalization  of a theorem by \'Etienne Ghys \cite{ghys}. Recall that the ratio of a triple 
$(u,v,w)\in \C^3$ is $\frac{u-v}{u-w}$.

\begin{thm}\label{thm:ghys}
Let $\psi:\C\to\cd$ be an injective holomorphic immersion, and $L=\psi(\C)$. Assume that  $(L_n)$ is a sequence
of immersed complex submanifolds converging to $L$ in the following sense: if  $K\Subset L$ 
 is any   relatively compact  subset (relative to the leafwise topology), 
 then $L_n$ contains a graph over a neighborhood of $K$ 
 for  large $n$, that is there exists  a neighborhood $N(K)$ of $K$ in $L$ and 
 a sequence of  injective  holomorphic maps $\pi_n: N(K)\to L_n$ such that $\pi_n(x)\to x$ 
 for every $x$. Assume further that for every $n$, $L_n$ is biholomorphic to $\C$. 
 
 Then the affine structures on the  $L_n$ converge  to that of $L$ in the following sense: for any 
 compact set $K\Subset L$ as above and any triple $(x,y,z)\in K^3$, if 
 $(x_n, y_n, z_n)\in \pi_n (N(K))$ are  
  close to $(\pi_n (x), \pi_n (y), \pi_n (z))$ 
    and 
  converge  to $(x,y,z)$, then the corresponding  ratios   converge as well. 
  \end{thm}

The point of this statement is to emphasize that there is no need in Ghys' theorem to 
work with the leaves of a Riemann surface lamination. Also, compactness of the ambient space is not required. 
The theorem is certainly not written in its most general form: one might assume more generally that 
\begin{itemize}
\item the $\pi_n$ are $(1+\e_n)$ quasi-conformal for some $\e_n\to 0$;
\item  $L$ and the $L_n$ are parabolic Riemann surfaces instead of copies of $\C$. 
\end{itemize}
The adaptation is left  to the reader. Notice also that any submanifold $V$ 
of a Stein manifold admits a neighborhood $W$ endowed with a holomorphic retraction $W\to V$ (see \cite[Cor. 1]{siu}). Therefore 
our convergence  assumption essentially means  that     $L_n$ converges to $L$ with multiplicity 1.

\begin{proof} We   follow \cite[\S 4]{ghys} closely. Pick a triple of distinct points  $(x,y,z)$ in $L$ and  $R_0$ 
  such that $\psi(D(0,R_0))$ contains $x,y,z$. For $\alpha  \in L$ let   $\widetilde \alpha = \psi\inv(\alpha)$.  Without loss of generality we may assume $R_0=1$. Let $R$ be a large positive number to be determined. For   $n\geq n(R)$,  $\pi_n$ is well defined in $\psi(D(0,R))$.  Let $(x_n, y_n, z_n)\in \pi_n(D(0,1))^3$
converging to $(x,y,z)$, and fix $\e>0$.. Then  by assumption 
$(\pi_n\inv(x_n), \pi_n\inv(y_n), \pi_n\inv(z_n))$ converges to $(x,y,z)$ for the leafwise topology
 in $L$. Let $\psi_n:\C\to L_n$ be any parameterization, and let 
 $\widetilde x_n = \psi_n\inv(x_n)$, $\widetilde y_n = \psi\inv(y_n)$ and 
  $\widetilde z_n = \psi\inv(z_n)$. Without loss of generality we may assume  $\widetilde x_n = 0$. 
   We have to show that for large $n$, the  ratio of $(\widetilde x_n, \widetilde y_n, 
  \widetilde z_n)$ is close to that of $(\widetilde x, \widetilde y, \widetilde z)$. 
  
  By assumption $h_n:=\psi_n\inv\circ \pi_n\circ \psi : D(0, R)\to \C$ is 
  an injective holomorphic map. By  
  renormalizing $\psi_n$ we may assume that $h_n'(0)=1$ (we use $L_n\simeq \C$ precisely here). Then by the Koebe distortion theorem, $h_n$ is almost affine in $D(0,1)$, that is, it 
   distorts the  ratios of points in $D(0,1)$ by some  small   amount $\e(R)$. Fix $R$ so large  that $\e(R)<\e$. In particular for $n\geq n(R)$ we get that 
   $$\abs{\frac{h_n(\widetilde x) - h_n(\widetilde y)}{h_n(\widetilde x) - h_n(\widetilde z)} -  \frac{\widetilde x-\widetilde y}{\widetilde x-\widetilde z} } \leq \e.$$   Now for $\alpha \in K$, $h_n(\widetilde \alpha)$
    is the parameter  in $\C$
   corresponding to $\pi_n(\alpha) \in L_n$, so $\widetilde \alpha_n$ is close to $h_n(\widetilde\alpha)$ in $\C$ and 
    for large $n$ we also get that
    $$\abs{\frac{h_n(\widetilde x) - h_n(\widetilde y)}{h_n(\widetilde x) - h_n(\widetilde z)} - \frac{\widetilde x_n-\widetilde y_n}{\widetilde x_n- \widetilde z_n} } \leq \e,$$ 
   and we are done. 
\end{proof}

 
 \bibliographystyle{plain}
 \bibliography{bib-structure}

\end{document}